\documentclass[12pt, letterpaper]{amsart}
\usepackage[utf8]{inputenc}
\usepackage{amsmath}
\usepackage{amsthm}
\usepackage{geometry}
\usepackage{amsfonts}
\usepackage[colorlinks=true, urlcolor=red, citecolor=cyan, linkcolor=blue]{hyperref}
\usepackage{changepage}
\usepackage{amssymb}
\usepackage{multicol}
\usepackage{tikz}
\usepackage{tikz-cd}
\usetikzlibrary{babel}
\usepackage{float}
\usepackage{multirow}
\usepackage{graphicx}
\usepackage{dsfont}
\usepackage{mathrsfs}
\usepackage[english]{babel}
\usepackage{pgfplots}
\pgfplotsset{compat=1.15}
\usepackage{mathrsfs}
\usetikzlibrary{arrows}

\newtheorem{teo}{Theorem}[section]
\newtheorem{cor}[teo]{Corollary}

\newtheorem{lem}[teo]{Lemma}
\theoremstyle{definition}
\newtheorem{defi}[teo]{Definition}

\newtheorem{exam}[teo]{Example}

\renewcommand{\cos}{\text{cos}}
\renewcommand{\sin}{\text{sin}}

\newcommand{\R}{\mathbb{R}}

\newcommand{\Z}{\mathbb{Z}}
\newcommand{\N}{\mathbb{N}}

\theoremstyle{definition}
\newcommand{\leb}{\text{Leb}}

\usepackage[shortlabels]{enumitem}

\usepackage[utf8]{inputenc}
\numberwithin{equation}{section}

\usepackage{etoolbox}

\makeatletter
\newcommand{\affiliation}[1]{\gdef\@affiliation{#1}}
\newcommand{\@affiliation}{}

\patchcmd{\@setauthors}
  {\endtrivlist}
  {\ifx\@affiliation\@empty\else
     \par\vskip3pt {\normalfont\footnotesize \@affiliation}\par
   \fi
   \endtrivlist}
  {}{\typeout{WARNING: Could not patch \string\@setauthors}}
\makeatother

\title{A nonlinear version of the $\alpha$-Kakutani equidistribution problem}

\author{I. Rojas Aravena}
\affiliation{Department of Mathematics, University of British Columbia}
\address{Department of Mathematics, University of British Columbia}
\email{i.andres@math.ubc.ca}%\date{2025}

\begin{document}

\begin{abstract}
In this work, we extend results of Kakutani; Adler and Flatto; Smilansky; Pollicott and Sewell on the equidistribution of endpoints generated by interval-splitting procedures. We study a nonlinear version of the problem generated by a finite or countable family of $\mathcal{C}^{1+ \varepsilon}$ contractions and prove Lebesgue equidistribution under suitable nonlattice and thermodynamic regularity assumptions.
\end{abstract}
%\tableofcontents

\maketitle

\section{Introduction}
In this work we study a nonlinear version of the Kakutani splitting procedure, originally introduced by Kakutani in \cite{KakutaniOriginal}; see also \cite{AdleryFlatto} for another proof and historical background.

Let us first recall the classical construction. A partition of $[0,1]$ will mean a finite or countable collection of closed subintervals with positive length, pairwise disjoint interiors, and whose union covers $[0,1]$ up to a set of Lebesgue measure zero.

\begin{defi}[Kakutani splitting procedure]
Let $\alpha\in (0,1)$ be a fixed real. Let $\mathcal{P}_0:= \{[0,1]\}$. Let $\mathcal{P}_{n+ 1}$ be the partition generated by taking all intervals of maximal length of $\mathcal{P}_n$ and subdividing them in two intervals of lengths proportional to $\alpha$ and $1-\alpha$ respectively.
\end{defi}

Consider $\widehat{E_n} = \bigcup_{I\in \mathcal{P}_n} \partial I$ the set of endpoints at the $n$-th stage of this process and let 
$$\widehat{\mu}_n = \frac{1}{\# \widehat{E}_n}\sum_{x \in \widehat{E}_n} \delta_x,$$
where $\delta_x$ denotes the Dirac delta measure at $x$.

\begin{exam}\label{0IntroExample2}
The following diagram shows the first seven partitions for the choice $\alpha = 2/5$. 
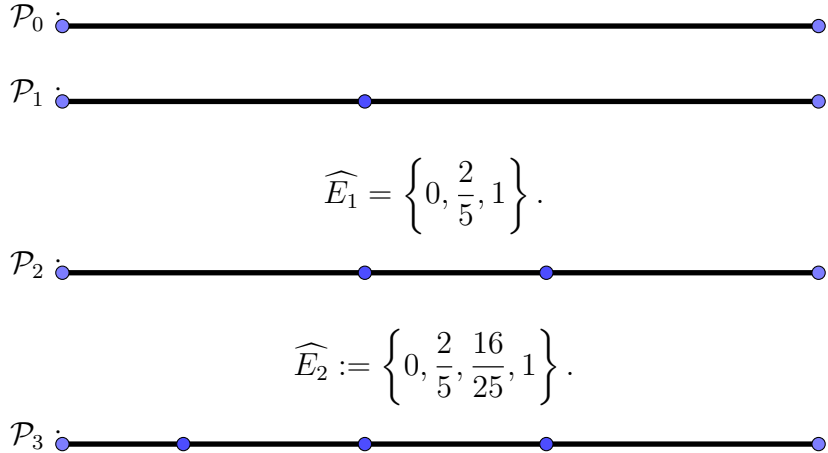
\begin{center}
\definecolor{xdxdff}{rgb}{0.49019607843137253,0.49019607843137253,1.}
\begin{tikzpicture}[line cap=round,line join=round,>=triangle 45,x=10.0cm,y=1.0cm][$\mathcal{P}_0:\quad $]
\clip(-0.06590059606070635,-0.5153678449999174) rectangle (1.1871596350345994,0.32666340089550305);
\draw [line width=2.pt] (0.,0.)-- (1.,0.);
\begin{scriptsize}
\draw [fill=xdxdff] (0.,0.) circle (2.5pt);
\draw [fill=xdxdff] (1.,0.) circle (2.5pt);
\end{scriptsize}
\end{tikzpicture}
\end{center}

\begin{center}
\definecolor{ududff}{rgb}{0.30196078431372547,0.30196078431372547,1.}
\definecolor{xdxdff}{rgb}{0.49019607843137253,0.49019607843137253,1.}
\begin{tikzpicture}[line cap=round,line join=round,>=triangle 45,x=10.0cm,y=1.0cm][$\mathcal{P}_1: \quad$]
\clip(-0.06590059606070635,-0.5153678449999174) rectangle (1.1871596350345994,0.32666340089550305);
\draw [line width=2.pt] (0.,0.)-- (1.,0.);
\begin{scriptsize}
\draw [fill=xdxdff] (0.,0.) circle (2.5pt);
\draw [fill=xdxdff] (1.,0.) circle (2.5pt);
\draw [fill=ududff] (0.4,0.) circle (2.5pt);
\end{scriptsize}
\end{tikzpicture}
\end{center}
$$\widehat{E_1} = \left\{0,\frac{2}{5},1\right\}.$$

\begin{center}
\definecolor{ududff}{rgb}{0.30196078431372547,0.30196078431372547,1.}
\definecolor{xdxdff}{rgb}{0.49019607843137253,0.49019607843137253,1.}
\begin{tikzpicture}[line cap=round,line join=round,>=triangle 45,x=10.0cm,y=1.0cm][$\mathcal{P}_2: \quad$]
\clip(-0.06590059606070635,-0.5153678449999174) rectangle (1.1871596350345994,0.32666340089550305);
\draw [line width=2.pt] (0.,0.)-- (1.,0.);
\begin{scriptsize}
\draw [fill=xdxdff] (0.,0.) circle (2.5pt);
\draw [fill=xdxdff] (1.,0.) circle (2.5pt);
\draw [fill=ududff] (0.4,0.) circle (2.5pt);
\draw [fill=ududff] (0.64,0.) circle (2.5pt);
\end{scriptsize}
\end{tikzpicture}
\end{center}
$$\widehat{E_2} :=\left\{0, \frac{2}{5}, \frac{16}{25},1 \right\}.$$

\begin{center}
\definecolor{ududff}{rgb}{0.30196078431372547,0.30196078431372547,1.}
\definecolor{xdxdff}{rgb}{0.49019607843137253,0.49019607843137253,1.}
\begin{tikzpicture}[line cap=round,line join=round,>=triangle 45,x=10.0cm,y=1.0cm][$\mathcal{P}_3: \quad$]
\clip(-0.06590059606070635,-0.5153678449999174) rectangle (1.1871596350345994,0.32666340089550305);
\draw [line width=2.pt] (0.,0.)-- (1.,0.);
\begin{scriptsize}
\draw [fill=xdxdff] (0.,0.) circle (2.5pt);
\draw [fill=xdxdff] (1.,0.) circle (2.5pt);
\draw [fill=ududff] (0.4,0.) circle (2.5pt);
\draw [fill=ududff] (0.64,0.) circle (2.5pt);
\draw [fill=ududff] (0.16,0.) circle (2.5pt);
\end{scriptsize}
\end{tikzpicture}
\end{center}
$$\widehat{E_3}:= \left\{0, \frac{4}{25}, \frac{2}{5}, \frac{16}{25}, 1\right\}.$$

\begin{center}
\definecolor{ududff}{rgb}{0.30196078431372547,0.30196078431372547,1.}
\definecolor{xdxdff}{rgb}{0.49019607843137253,0.49019607843137253,1.}
\begin{tikzpicture}[line cap=round,line join=round,>=triangle 45,x=10.0cm,y=1.0cm][$\mathcal{P}_4:\quad$]
\clip(-0.06590059606070635,-0.5153678449999174) rectangle (1.1871596350345994,0.32666340089550305);
\draw [line width=2.pt] (0.,0.)-- (1.,0.);
\begin{scriptsize}
\draw [fill=xdxdff] (0.,0.) circle (2.5pt);
\draw [fill=xdxdff] (1.,0.) circle (2.5pt);
\draw [fill=ududff] (0.4,0.) circle (2.5pt);
\draw [fill=ududff] (0.64,0.) circle (2.5pt);
\draw [fill=ududff] (0.16,0.) circle (2.5pt);
\draw [fill=ududff] (0.784,0.) circle (2.5pt);
\end{scriptsize}
\end{tikzpicture}
\end{center}
$$\widehat{E_4}:= \left\{0, \frac{4}{25}, \frac{2}{5}, \frac{16}{25}, \frac{98}{125},1\right\}.$$

\begin{center}
\definecolor{ududff}{rgb}{0.30196078431372547,0.30196078431372547,1.}
\definecolor{xdxdff}{rgb}{0.49019607843137253,0.49019607843137253,1.}
\begin{tikzpicture}[line cap=round,line join=round,>=triangle 45,x=10.0cm,y=1.0cm][$\mathcal{P}_5: \quad$]
\clip(-0.06590059606070635,-0.5153678449999176) rectangle (1.1871596350345994,0.3266634008955031);
\draw [line width=2.pt] (0.,0.)-- (1.,0.);
\begin{scriptsize}
\draw [fill=xdxdff] (0.,0.) circle (2.5pt);
\draw [fill=xdxdff] (1.,0.) circle (2.5pt);
\draw [fill=ududff] (0.4,0.) circle (2.5pt);
\draw [fill=ududff] (0.64,0.) circle (2.5pt);
\draw [fill=ududff] (0.16,0.) circle (2.5pt);
\draw [fill=ududff] (0.784,0.) circle (2.5pt);
\draw [fill=ududff] (0.256,0.) circle (2.5pt);
\draw [fill=ududff] (0.8608,0.) circle (2.5pt);
\draw [fill=ududff] (0.496,0.) circle (2.5pt);
\end{scriptsize}
\end{tikzpicture}
\end{center}
$$\widehat{E_5} := \left\{0, \frac{4}{25}, \frac{32}{125} ,\frac{2}{5}, \frac{62}{125}, \frac{16}{25}, \frac{98}{125}, 1\right\}.$$

\begin{center}
\definecolor{ududff}{rgb}{0.30196078431372547,0.30196078431372547,1.}
\definecolor{xdxdff}{rgb}{0.49019607843137253,0.49019607843137253,1.}
\begin{tikzpicture}[line cap=round,line join=round,>=triangle 45,x=10.0cm,y=1.0cm][$\mathcal{P}_6: \quad$]
\clip(-0.0604438322896736,-0.4862832941003133) rectangle (1.1275717746549223,0.31203927398913556);
\draw [line width=2.pt] (0.,0.)-- (1.,0.);
\begin{scriptsize}
\draw [fill=xdxdff] (0.,0.) circle (2.5pt);
\draw [fill=xdxdff] (1.,0.) circle (2.5pt);
\draw [fill=ududff] (0.4,0.) circle (2.5pt);
\draw [fill=ududff] (0.64,0.) circle (2.5pt);
\draw [fill=ududff] (0.16,0.) circle (2.5pt);
\draw [fill=ududff] (0.784,0.) circle (2.5pt);
\draw [fill=ududff] (0.256,0.) circle (2.5pt);
\draw [fill=ududff] (0.8608,0.) circle (2.5pt);
\draw [fill=ududff] (0.496,0.) circle (2.5pt);
\end{scriptsize}
\end{tikzpicture}
\end{center}
$$\widehat{E_6} := \left\{0, \frac{4}{25}, \frac{32}{125} ,\frac{2}{5}, \frac{62}{125}, \frac{16}{25}, \frac{98}{125}, \frac{544}{625}, 1\right\}.$$

\begin{center}
\definecolor{ududff}{rgb}{0.30196078431372547,0.30196078431372547,1.}
\definecolor{xdxdff}{rgb}{0.49019607843137253,0.49019607843137253,1.}
\begin{tikzpicture}[line cap=round,line join=round,>=triangle 45,x=10.0cm,y=1.0cm][$\mathcal{P}_7: \quad$]
\clip(-0.06590059606070635,-0.5153678449999176) rectangle (1.1871596350345994,0.3266634008955031);
\draw [line width=2.pt] (0.,0.)-- (1.,0.);
\begin{scriptsize}
\draw [fill=xdxdff] (0.,0.) circle (2.5pt);
\draw [fill=xdxdff] (1.,0.) circle (2.5pt);
\draw [fill=ududff] (0.4,0.) circle (2.5pt);
\draw [fill=ududff] (0.64,0.) circle (2.5pt);
\draw [fill=ududff] (0.16,0.) circle (2.5pt);
\draw [fill=ududff] (0.784,0.) circle (2.5pt);
\draw [fill=ududff] (0.256,0.) circle (2.5pt);
\draw [fill=ududff] (0.8608,0.) circle (2.5pt);
\draw [fill=ududff] (0.496,0.) circle (2.5pt);
\draw [fill=ududff] (0.064,0.) circle (2.5pt);
\end{scriptsize}
\end{tikzpicture}
\end{center}
$$\widehat{E_7}:= \left\{0, \frac{8}{125}, \frac{4}{25}, \frac{32}{125} ,\frac{2}{5}, \frac{62}{125}, \frac{16}{25}, \frac{98}{125}, \frac{544}{625}, 1\right\}$$

\end{exam}

\begin{teo}[Kakutani, \cite{KakutaniOriginal}]
For all $\alpha \in (0,1)$, the set $E_n$ becomes uniformly distributed as $n\to \infty$, i.e., for all intervals $I \subset [0,1]$
$$\lim_{n \to \infty} \mu_n(I) = \leb(I),$$
where $\leb$ denotes the Lebesgue measure on [0,1].
\end{teo}

A natural generalization of this result was first introduced by Volčic in \cite{Volcic}. In his work, instead of splitting all maximal intervals in two subintervals, he split them according a fixed finite partition consisting of $N \geq 2$ subintervals. In other words, at each stage, one splits all intervals of maximal length into $N$ pieces whose lengths have a certain fixed proportion.

\begin{exam}\label{0introExampleFinitos}
The following diagrams show the first 5 partitions for the initial partition $\{[0, 1/2], [1/2, 4/5], [4/5, 1]\}.$

\begin{center}
\definecolor{xdxdff}{rgb}{0.49019607843137253,0.49019607843137253,1.}
\begin{tikzpicture}[line cap=round,line join=round,>=triangle 45,x=10.0cm,y=1.0cm][$\mathcal{P}_0:\quad $]
\clip(-0.06590059606070635,-0.5153678449999174) rectangle (1.1871596350345994,0.32666340089550305);
\draw [line width=2.pt] (0.,0.)-- (1.,0.);
\begin{scriptsize}
\draw [fill=xdxdff] (0.,0.) circle (2.5pt);
\draw [fill=xdxdff] (1.,0.) circle (2.5pt);
\end{scriptsize}
\end{tikzpicture}
\end{center}

\begin{center}
\definecolor{xdxdff}{rgb}{0.49019607843137253,0.49019607843137253,1.}
\begin{tikzpicture}[line cap=round,line join=round,>=triangle 45,x=10.0cm,y=1.0cm][$\mathcal{P}_1:\quad $]
\clip(-0.06590059606070635,-0.5153678449999174) rectangle (1.1871596350345994,0.32666340089550305);
\draw [line width=2.pt] (0.,0.)-- (1.,0.);
\begin{scriptsize}
\draw [fill=xdxdff] (0.,0.) circle (2.5pt);
\draw [fill=xdxdff] (1.,0.) circle (2.5pt);
\draw [fill=xdxdff] (0.5,0.) circle (2.5pt);
\draw [fill=xdxdff] (0.8,0.) circle (2.5pt);
\end{scriptsize}
\end{tikzpicture}
\end{center}
$$\widehat{E_1} = \left\{0, \frac{1}{2}, \frac{4}{5}, 1\right\}$$

\begin{center}
\definecolor{xdxdff}{rgb}{0.49019607843137253,0.49019607843137253,1.}
\begin{tikzpicture}[line cap=round,line join=round,>=triangle 45,x=10.0cm,y=1.0cm][$\mathcal{P}_2:\quad $]
\clip(-0.06590059606070635,-0.5153678449999174) rectangle (1.1871596350345994,0.32666340089550305);
\draw [line width=2.pt] (0.,0.)-- (1.,0.);
\begin{scriptsize}
\draw [fill=xdxdff] (0.,0.) circle (2.5pt);
\draw [fill=xdxdff] (1.,0.) circle (2.5pt);
\draw [fill=xdxdff] (0.5,0.) circle (2.5pt);
\draw [fill=xdxdff] (0.8,0.) circle (2.5pt);
\draw [fill=xdxdff] (0.25,0.) circle (2.5pt);
\draw [fill=xdxdff] (0.4,0.) circle (2.5pt);
\end{scriptsize}
\end{tikzpicture}
\end{center}
$$\widehat{E_2} = \left\{0, \frac{1}{4}, \frac{2}{5}, \frac{1}{2}, \frac{4}{5}, 1\right\}$$

\begin{center}
\definecolor{xdxdff}{rgb}{0.49019607843137253,0.49019607843137253,1.}
\begin{tikzpicture}[line cap=round,line join=round,>=triangle 45,x=10.0cm,y=1.0cm][$\mathcal{P}_3:\quad $]
\clip(-0.06590059606070635,-0.5153678449999174) rectangle (1.1871596350345994,0.32666340089550305);
\draw [line width=2.pt] (0.,0.)-- (1.,0.);
\begin{scriptsize}
\draw [fill=xdxdff] (0.,0.) circle (2.5pt);
\draw [fill=xdxdff] (1.,0.) circle (2.5pt);
\draw [fill=xdxdff] (0.5,0.) circle (2.5pt);
\draw [fill=xdxdff] (0.8,0.) circle (2.5pt);
\draw [fill=xdxdff] (0.25,0.) circle (2.5pt);
\draw [fill=xdxdff] (0.4,0.) circle (2.5pt);
\draw [fill=xdxdff] (0.65,0.) circle (2.5pt);
\draw [fill=xdxdff] (0.74,0.) circle (2.5pt);
\end{scriptsize}
\end{tikzpicture}
\end{center}
$$\widehat{E_3} = \left\{0, \frac{1}{4}, \frac{2}{5}, \frac{1}{2}, \frac{13}{20}, \frac{37}{50}, \frac{4}{5}, 1\right\}$$

\begin{center}
\definecolor{xdxdff}{rgb}{0.49019607843137253,0.49019607843137253,1.}
\begin{tikzpicture}[line cap=round,line join=round,>=triangle 45,x=10.0cm,y=1.0cm][$\mathcal{P}_4:\quad $]
\clip(-0.06590059606070635,-0.5153678449999174) rectangle (1.1871596350345994,0.32666340089550305);
\draw [line width=2.pt] (0.,0.)-- (1.,0.);
\begin{scriptsize}
\draw [fill=xdxdff] (0.,0.) circle (2.5pt);
\draw [fill=xdxdff] (1.,0.) circle (2.5pt);
\draw [fill=xdxdff] (0.5,0.) circle (2.5pt);
\draw [fill=xdxdff] (0.25,0.) circle (2.5pt);
\draw [fill=xdxdff] (0.4,0.) circle (2.5pt);
\draw [fill=xdxdff] (0.65,0.) circle (2.5pt);
\draw [fill=xdxdff] (0.74,0.) circle (2.5pt);
\draw [fill=xdxdff] (0.8,0.) circle (2.5pt);
\draw [fill=xdxdff] (0.125,0.) circle (2.5pt);
\draw [fill=xdxdff] (0.2,0.) circle (2.5pt);
\end{scriptsize}
\end{tikzpicture}
\end{center}
$$\widehat{E_4} = \left\{0, \frac{1}{8}, \frac{1}{5}, \frac{1}{4}, \frac{2}{5}, \frac{1}{2}, \frac{13}{20}, \frac{37}{50}, \frac{4}{5}, 1\right\}$$

\begin{center}
\definecolor{xdxdff}{rgb}{0.49019607843137253,0.49019607843137253,1.}
\begin{tikzpicture}[line cap=round,line join=round,>=triangle 45,x=10.0cm,y=1.0cm][$\mathcal{P}_5:\quad $]
\clip(-0.06590059606070635,-0.5153678449999174) rectangle (1.1871596350345994,0.32666340089550305);
\draw [line width=2.pt] (0.,0.)-- (1.,0.);
\begin{scriptsize}
\draw [fill=xdxdff] (0.,0.) circle (2.5pt);
\draw [fill=xdxdff] (1.,0.) circle (2.5pt);
\draw [fill=xdxdff] (0.5,0.) circle (2.5pt);
\draw [fill=xdxdff] (0.25,0.) circle (2.5pt);
\draw [fill=xdxdff] (0.4,0.) circle (2.5pt);
\draw [fill=xdxdff] (0.65,0.) circle (2.5pt);
\draw [fill=xdxdff] (0.74,0.) circle (2.5pt);
\draw [fill=xdxdff] (0.8,0.) circle (2.5pt);
\draw [fill=xdxdff] (0.125,0.) circle (2.5pt);
\draw [fill=xdxdff] (0.2,0.) circle (2.5pt);
\draw [fill=xdxdff] (0.8,0.) circle (2.5pt);
\draw [fill=xdxdff] (0.9,0.) circle (2.5pt);
\draw [fill=xdxdff] (0.96,0.) circle (2.5pt);
\end{scriptsize}
\end{tikzpicture}
\end{center}
$$\widehat{E_5} = \left\{0, \frac{1}{8}, \frac{1}{5}, \frac{1}{4}, \frac{2}{5}, \frac{1}{2}, \frac{13}{20}, \frac{37}{50}, \frac{4}{5}, \frac{9}{10}, \frac{24}{25}, 1\right\}$$
\end{exam}

Several other questions and generalizations related to Kakutani's procedure have been studied. Discrepancy estimates were considered in \cite{CarboneLS, DrmotaInfusino}. Higher-dimensional analogues were studied in \cite{KakutaniInRd} and \cite{Smilansky}. The effect of beginning with an initial partition different from $\{[0,1]\}$ was investigated in \cite{AistleitnerHofer}.

The result most closely related to the present work is due to Pollicott and Sewell \cite{PollicottSewell}. They considered an infinite version of the Kakutani splitting procedure. In their setting, one starts with a countable partition $\mathcal{P}$, and at each stage all intervals of maximal length are subdivided according to this same partition. Instead of considering all endpoints in the current partition, they consider the set of endpoints of intervals which have already been split. More precisely, if $\{\mathcal{P}_n\}_{n \geq 0}$ is the resulting sequence of partitions, they define
$${E_n} := \left\{\min(I), \max(I): I \in \bigcup_{i = 0}^n \mathcal{P}_i\setminus \mathcal{P}_{i+1}\right\}.$$
For simplicity, they restrict their attention to the corresponding set \(L_n\) of left endpoints.

%Recall that $\text{Leb}$ be the Lebesgue measure on $[0,1]$.

Let $\mathcal{I}$ be an at most countable set. For each $i\in \mathcal{I}$, consider $T_i \colon [0,1] \to I_i\subset [0,1]$ smooth bijective contractions. And suppose that for each $i\neq j \in \mathcal{I}$ we have $\text{int}(I_i) \cap \text{int}(I_j) = \emptyset$ and $\leb \left(\bigcup_{i \in \mathcal{I}} T_i ([0,1])\right) = 1$. Associated to the system $\{T_i\}_{i \in \mathcal{I}}$ there is a full-branch expanding map $f := f_\mathcal{P}$, defined by the inverse branches $\{T_i^{-1}\}$. Let $\Sigma$ be the space of all sequences taking values in the alphabet $\mathcal{I}$, i.e.,
$$\Sigma:= \mathcal{I}^\N = \{(x_1, x_2, \ldots) : x_j \in \mathcal{I}, \text{ for all }j\geq 1\}.$$
We endow $\mathcal{I}$ with the discrete topology and $\Sigma$ with the product topology. Let $\sigma \colon \Sigma \to \Sigma$ be the left shift. Consider $\pi\colon \Sigma \to [0,1]$ be the coding map satisfying
$$f\circ \pi = \pi \circ \sigma.$$
The geometric potential associated to $f$ is defined as
$$\gamma_{\mathcal{P}}(\underline{x}) = \log f'(\pi(\underline{x}))$$

The arithmetic properties of this potential play an essential role. Our main result assumes that $\gamma_\mathcal{P}$  is nonlattice, which means that the potential is not equal to another potential taking values in a discrete subgroup of $\R$ plus a coboundary. This will be properly defined in Definition \ref{1.2LatticeDefinition}.

\begin{teo}\label{IntroMainTheorem}
Let $\mathcal{I}$ be an at most countable set of indices, and let $\mathcal{P}$ be the partition generated by a family $\{T_i\}_{i\in \mathcal{I}}$ satisfying the following assumptions:
\begin{itemize}
    \item each $T_i$ is a strictly increasing $\mathcal{C}^1$ diffeomorphism from $[0,1]$ onto its image, and for all $i\in \mathcal{I}$ the functions $T_i^{-1}\in \mathcal{C}^{1 + \varepsilon}$ for some \(\varepsilon\in(0,1)\);
    \item there exists \(c_M\in(0,1)\) such that
    $$0 < T_i'(x) < c_M < 1,$$
    for every $i\in \mathcal{I}$ and every $x\in [0,1]$;
    \item the geometric potential $\gamma_\mathcal{P}$ is nonlattice (Definition \ref{1.2LatticeDefinition}) regular potential (Definition \ref{1.2RegularPotential}).
\end{itemize}
Let $\mathcal{P}_0 = \{[0,1]\}$, $\mathcal{P}_1 = \{T_i[0,1]\}_{i \in \mathcal{I}}$ and for each $n$ define recursively $\mathcal{P}_{n+1}$ as the partition obtained by taking each interval of maximal length $T_v([0,1])$ of $\mathcal{P}_n$ and replacing them with the intervals $\{(T_v\circ T_i)([0,1])\}_{i \in \mathcal{I}}$. Let $\{L_n\}_{n \geq 0}$ be the sequence of left endpoints of intervals which have been split up to time $n$, and define
$$\mu_n = \frac{1}{\#L_n} \sum_{x\in L_n}\delta_x.$$
Then $L_n$ becomes uniformly distributed with respect to Lebesgue measure. That is, for every interval $J\subset [0,1]$,
$$\lim_{n\to \infty} \mu_n(J) = \leb(J).$$
\end{teo}

The proof follows the general strategy of Pollicott and Sewell, but several new difficulties appear in the nonlinear setting. First, the length of a cylinder \(T_v[0,1]\) is no longer exactly multiplicative in the word \(v\). Instead, it must be compared with Birkhoff sums of the geometric potential \(\gamma_{\mathcal P}\). Second, the symbolic counting functions which arise in the proof are governed by renewal theorems for subshifts of finite type. We use renewal theorems of Lalley \cite{Lalley} and \cite{InfiniteRenewal} to obtain the required asymptotics.

The lattice case has a different nature. In the finite alphabet setting, we prove a characterization showing that if the geometric potential is lattice, then the corresponding nonlinear system is smoothly conjugate to an affine model. More precisely, the geometric potential is cohomologous to a function which is constant on cylinders of length one. Nevertheless, this characterization does not make the lattice case a formal consequence of the affine case. The lattice renewal asymptotics retain the residue class of the logarithmic scale parameter, and this produces additional fractional part terms. We give an example showing that these terms are not invariant under smooth conjugacy.

This paper is organized as follows. In Section \ref{SectionPrelim}, we introduce the notation, the main result, and recall the results from thermodynamic formalism and renewal theory that will be used. In Section \ref{SectionProof} we prove Theorem \ref{IntroMainTheorem}. Finally in Section \ref{SectionLattice} we study the lattice case: we prove a finite alphabet characterization of lattice geometric potentials as smooth conjugates of affine models, and explain why the lattice renewal asymptotics do not lead directly to the nonlattice equidistribution conclusion.

\subsection*{Acknowledgment}
I am grateful to Godofredo Iommi for proposing the problem. And to Brian Marcus and Pablo Shmerkin for their useful comments and guidance. In particular, I thank Pablo Shmerkin for suggesting the characterization of lattice potentials. Finally, thanks to the unknown referee for their useful comments in the proof of Theorem \ref{2RTNonLatticeNormalRenewalTheorem}.

The author used ChatGPT (GPT-5.5 Thinking, OpenAI, May 2026) to establish Example \ref{4Theorem: Example4.1}. The computation and all details were subsequently checked and verified by the author.

\section{Preliminaries}\label{SectionPrelim}
\subsection{Notation and main result}
%Objective: Establish notation for express everything as composing transformations.

\begin{defi}\label{1.1DefPartition}
A partition of $[0,1]$ is a finite or countable collection 
$\mathcal{P} = \{I_j\}_{j \in \mathcal{I}}$ of closed subintervals of $[0,1]$, each of positive length such that 
\begin{itemize}[$\bullet$]
    \item for $i\neq j$, $\text{int}(I_i) \cap \text{int}(I_j) = \varnothing$.
    \item $\sum_{j \in \mathcal{I}} \leb(I_j)= 1$.
\end{itemize}
\end{defi}

\begin{defi}\label{1.1DefPartitionGeneratedbyTi}
Let $\mathcal{I}$ be a set of indices at most countable and consider $\{T_i\}_{i\in \mathcal{I}}$, a collection of functions from $[0,1] \to [0,1]$ such that 
\begin{itemize}
    \item for all $i \in \mathcal{I}$, $T_i [0,1]$ is a closed subinterval of $[0,1]$;
    \item for $i \neq j$
    $$T_i([0,1)) \cap T_j([0,1)) = \varnothing;$$
    \item the sum of the length of intervals $T_i[0,1)$ is 1, i.e.,
    $$\sum_{i\in \mathcal{I}} \leb(T_i [0,1 )) = 1.$$
    \item each $T_i$ is a strictly increasing $\mathcal{C}^1$ diffeomorphism from $[0,1]$ onto $T_i[0,1]$ and the derivative of the inverse function of $T_i$ is $\varepsilon$-Hölder continuous and positive, i.e., $T_i^{-1}\in \mathcal{C}^{1 + \varepsilon}$ for every $i \in \mathcal{I}$, for some $\varepsilon \in (0,1)$ and $T_i'(x) > 0$ for all $x\in [0,1]$;
    \item there exists $c_M \in (0,1)$ such that for each $i \in \mathcal{I}$ and all $x\in [0,1]$, $0 <T_i'(x) < c_M < 1$.
    \end{itemize}
Note that $\{T_i[0,1]\}_{i\in \mathcal{I}}$ is a partition in the sense of Definition \ref{1.1DefPartition}. We call $\{T_{i}[0,1]\}_{i\in \mathcal{I}}$ the partition generated by the functions $\{T_i\}_{i \in \mathcal{I}}$.
\end{defi}

\begin{exam}\label{1.1ExampleKakutani}
Given any $\alpha\in (0,1)$, the partition $\{[0,\alpha], [\alpha, 1]\}$ is the one generated by $T_1\colon [0,1] \to [0,\alpha]$ given by $x\mapsto \alpha x$ and $T_2\colon [0,1] \to [\alpha,1]$ given by $x\mapsto (1- \alpha )x + \alpha$.
\end{exam}

Let $\mathcal{P}$ be the partition generated by $\{T_i\}_{i \in \mathcal{I}}$.

\begin{defi}[Kakutani splitting procedure]
Let $\mathcal{Q}$ be the partition generated by $\{S_j\}_{j \in \mathcal{J}}$, for $S_j \colon [0,1] \to I_j \subset [0,1]$ satisfying the conditions of Definition \ref{1.1DefPartitionGeneratedbyTi} or $\{S_j\}_{j \in \mathcal{J}} = \{\text{Id}\}$. The $\{T_i\}_{i \in \mathcal{I}}$-refinement of $\mathcal{Q}$ is the partition obtained by taking the intervals of largest Lebesgue measure $\{S_k[0,1]\}_{k\in \mathcal{K}}$ and replacing them with the subintervals 
$$\{(S_k \circ T_i)[0,1]\}_{i\in \mathcal{I}, k \in \mathcal{K}}.$$
\end{defi}

\begin{defi}\label{1.1DefinSchemeProcess}
The scheme of interval substitution generated by $\{T_i\}_{i\in \mathcal{I}}$ is the sequence of partitions $\{\mathcal{P}_n\}_{n = 0}^{\infty}$ defined by 
\begin{itemize}
    \item $\mathcal{P}_0$ is the trivial partition (Generated by the identity).
    \item $\mathcal{P}_{n + 1}$ is the $(\{T_i\}_{i \in \mathcal{I}})$-refinement of $\mathcal{P}_n$.
\end{itemize}
\end{defi}

\begin{exam}
The scheme of interval substitution studied in \cite{KakutaniOriginal} and \cite{AdleryFlatto} corresponds to the partition generated by the two maps $\{T_1, T_2\}$ defined in Example \ref{1.1ExampleKakutani}.
\end{exam}

\begin{exam}
The regularity assumptions above cannot simply be omitted. For instance, $T_1 (x) = \sqrt{x}/2$ and $T_2(x) = (x+ 1)/2$. Both maps preserves the interval $[1/4, 1]$, then $T_v(0) \in \{0\}\cup[1/4, 1]$ for every finite word $v$. See also \cite{PollicottSewell}. 
\end{exam}

We now associate to the sequence of partitions $\{\mathcal{P}_n\}_{n = 0}^\infty$ a sequence of families of left endpoints of split intervals.

\begin{defi}
Given a scheme of interval substitution $\{\mathcal{P}_n\}_{n\in \N}$ generated by the maps $\{T_i\}_{i\in \mathcal{I}}$, define for all $n\in \N$, let
$$\widehat{L}_n := \{\min (I): I \in \mathcal{P}_{n + 1}\}.$$
Note that when $\mathcal{I}$ is infinite, then $\widehat{L}_n$ is also infinite. This is not useful if we want to define a probability measure uniformly distributed on $\widehat{L}_n$. To fix this, and following \cite{PollicottSewell}, we define
$$L_n = \{\min(I): I \in \mathcal{P}_k \setminus \mathcal{P}_{k + 1} \text{ for some } 0 \leq k \leq n\}.$$
This set can be understood as the set of left endpoints of intervals which have been split and it is always finite. 
%In the finite case the equidistribution of $L_n$ implies the equidistribution of $\widehat{L}_n$.
\end{defi}

Let 
\begin{equation}\label{1.1DefMeasures}
    \mu_n = \frac{1}{\# L_n} \sum_{x\in L_n}\delta_x,
\end{equation}
where $\delta_x$ is the measure concentrated at the point $x$.

%When $\mathcal{I}$ is finite, we define 
%$$\widehat{\mu}_n = \frac{1}{\# \widehat{L}_n} \sum_{x\in \widehat{ L_n}}\delta_x.$$

Volčič in \cite{Volcic} proved that when $\mathcal{I}$ is finite, and each transformation $\{T_i\}_{i \in \mathcal{I}}$ is an affine transformation, the set $\widehat{L}_n$ becomes equidistributed as $n\to \infty$. Pollicott and Sewell in \cite{PollicottSewell} under the hypothesis of finite entropy of the partition $\{T_i [0,1]\}_{i \in \mathcal{I}}$ with the Lebesgue measure, i.e.,
$$-\sum_{i\in \mathcal{I}}\alpha_i \log \alpha_i < \infty,$$
where $\alpha_i = \leb(T_i[0,1])$. They proved that the set $L_n$ becomes Lebesgue equidistributed as $n \to\infty$. The main objective of this work is to extend these equidistribution results to nonlinear interval substitution schemes.

By construction, the maps $\{T_i\}_{i\in \mathcal{I}}$ define an iterated function system whose attractor is $[0,1]$. Equivalently, the inverse branches $\{T_{i}^{-1}\}_{i \in \mathcal{I}}$ define a full-branch expanding map $f_\mathcal{P}$ on $[0,1]$, up to the ambiguity at partition endpoints or a set of Lebesgue measure zero, where we define $f_\mathcal{P}$ arbitrarily, i.e.
\begin{equation}\label{1.1DefiSistemaDinamico0,1}
    f_\mathcal{P}(x) = \begin{cases}
    0 &\text{if } x=0\\
    T_{i}^{-1}(x) &\text{if } x\in T_i([0,1))\\
    1 &\text{otherwise.}
\end{cases}
\end{equation}
When the partition is understood, we will use $f$ instead of $f_\mathcal{P}$.

Let $\Sigma$ be $\mathcal{I}^\N$ endowed with the product topology and let $\sigma\colon \Sigma \to \Sigma$ be the left shift. We will see that there exists a continuous function $\pi\colon \Sigma \to [0,1]$ such that 
$$f_\mathcal{P}\circ \pi = \pi \circ \sigma.$$
We define the geometric potential associated to $f_\mathcal{P}$ as $\gamma_{\mathcal{P}}\colon \Sigma \to \R$ by 
$$\gamma_\mathcal{P}(\underline{x}) = -\log T_{x_1}'(\pi(\sigma \underline{x})).$$
Again, when the partition is understood, we will use $\gamma$ instead of $\gamma_{\mathcal{P}}$.

\begin{defi}
Let $\gamma\colon \Sigma\to \R$. We say that $\gamma$ is lattice if there exists a continuous function $\psi \colon \Sigma \to \R$, a real number $a> 0$ and a continuous function $\zeta\colon \Sigma \to a\Z$ such that 
$$\gamma - \zeta = \psi - \psi \circ \sigma.$$
Otherwise, $\gamma$ is called nonlattice.
\end{defi}

The following theorem is the main result that we are going to prove in this work.

\begin{teo}\label{1MainTheo}
Let $\mathcal{I}$ be an at most countable set of indices, and let $\mathcal{P}$ be the partition generated by a family of functions $\{T_i\}_{i \in \mathcal{I}}$ satisfying the assumptions of Definition \ref{1.1DefPartitionGeneratedbyTi}. Suppose moreover that the geometric potential $\gamma_\mathcal{P}\colon \Sigma \to \R$, defined by 
$$\gamma_\mathcal{P}(\underline{x}) = -\log T_{x_1}'(\pi(\sigma\underline{x})),$$
is nonlattice and is a regular potential (Definition \ref{1.2RegularPotential}). Then the sequence $\{L_n\}$ of left endpoints of split intervals is Lebesgue equidistributed, that is, for every subinterval $J\subset [0,1]$
$$\lim_{n \to\infty} \mu_{n}(J) = \text{Leb}(J).$$
%Furthermore, if $\mathcal{I}$ is finite,
%$$\lim_{n \to \infty} \widehat{\mu}_n(I) = \leb(I).$$
\end{teo}

%When $\mathcal{I}$ is finite, this implies the equidistribution of $\widehat{L}_n$, generalizing the original problem of Kakutani to a nonlinear version of it.

For clarity, we use a reindexing of $L_n$. Note that for each $\lambda \in (0,1)$ there exists $n \in \N$ such that all the intervals $I$ in $\mathcal{P}_{n + 1}$ satisfy 
$$\leb(I) < \lambda.$$
This follows from the uniform contraction assumption. This motivates the following definition.

\begin{defi}
For $\lambda \in (0,1]$, let $n(\lambda)$ be the largest integer $n$ such that there exists $I\in \mathcal{P}_n$ with
$$\text{Leb}(I) \geq \lambda.$$
Equivalently, $\mathcal{P}_{n(\lambda) + 1}$ is the first partition in the process such that every interval has Lebesgue measure strictly less than $\lambda$.
\end{defi}

Let 
$$X_\lambda = L_{n(\lambda)}.$$
In a similar way, when $\mathcal{I}$ is finite, let 
$$\widehat{X}_\lambda = \widehat{L}_{n(\lambda)}.$$

To simplify notation for the composition of words we will define the following set.

\begin{defi}
Given $\mathcal{I}$, define the set $\Sigma_*$ as the set of all finite words from the alphabet $\mathcal{I}$, i.e.,
$$\Sigma_* = \bigcup_{n = 0}^\infty \mathcal{I}^n,$$
with $\mathcal{I}^0 := \{\emptyset\}$, where $\emptyset$ is the notation used for the empty sequence. For each $n\in \N$ we define the subset 
$$\Sigma_n = \mathcal{I}^n,$$
i.e., words of length $n$. For a word $v_1v_2\ldots v_n \in \Sigma_*$, we denote 
$$|v| = n.$$
\end{defi}

For $v = v_1 v_2\ldots v_n\in \Sigma_*$, we write
$$T_v = T_{v_1}\circ T_{v_2} \circ \ldots \circ T_{v_n},$$
and
$$\alpha_v = \leb(T_v[0,1)).$$
We set $T_\emptyset = \text{Id}$ and $\alpha_{\emptyset} = 1$.

Using this notation, it is possible to rewrite the sets $\widehat{X}_\lambda$ and $X_\lambda$  as
$$\widehat{X}_\lambda:= \{T_{vj}(0): v\in \Sigma_*, \alpha_v \geq \lambda, \, j\in \mathcal{I}\} \quad \text{ and } \quad X_\lambda = \{T_v(0): v\in \Sigma_*, \alpha_v \geq \lambda\}.$$

For $\lambda \in (0,1)$, define 
$$\mu_\lambda = \frac{1}{\# X_\lambda} \sum_{x\in X_\lambda} \delta_x,$$
and when $\mathcal{I}$ is finite,
$$\widehat{\mu}_\lambda=\frac{1}{\# \widehat{X}_\lambda} \sum_{x\in \widehat{X}_\lambda} \delta_x.$$

Our objective is to prove that for every interval $I\subset [0,1]$
$$\lim_{\lambda \to 0^+} \mu_\lambda(I) = \leb(I),$$
and when $\mathcal{I}$ is finite:
$$\lim_{\lambda \to 0^+}\widehat{\mu}_\lambda(I) = \leb(I).$$

\begin{defi}
For $\lambda \in (0,1]$ let 
$$A_\lambda:= \{v\in \Sigma_*: \alpha_v \geq \lambda\}.$$
\end{defi}

The strategy of the proof is to study the symbolic counting function associated to $A_\lambda$ as $\lambda\to 0^+$ with the help of a Renewal Theorem given in \cite{InfiniteRenewal}, motivated by the work of \cite{Lalley}. We then relate this asymptotic behaviour to the sets $X_\lambda$. 

We will define a dynamical system on the set $\Sigma\cup \Sigma_*$; for this reason we will need the shift extended to finite words.

Define
$$\sigma_*(x_1, \ldots, x_n) = (x_2, \ldots, x_n) \quad \text{ and }\quad \sigma_*(\emptyset) = \emptyset.$$
Let 
$$Y = \Sigma \cup \{(x_1, \ldots, x_n, 0, 0,\ldots, ): n\geq 0, x_i\in \mathcal{I}\}.$$
Then $Y\subset (\mathcal{I} \cup \{0\})^\N$. The corresponding transition matrix $A$ is given for all $i,j\in \mathcal{I}$, by
$$A(i,j ) = A(i, 0) =A(0,0) = 1, \quad \text{ and }\quad A(0,i) = 0. $$
With this construction, sequences in $\Sigma_*$ may be extended to infinite sequences by adjoining an additional symbol $0$ to the alphabet $\mathcal{I}$ and making the correspondence 
$$(x_1,\ldots, x_n) = ( x_1, \ldots, x_n, 0, 0, \ldots).$$
Thus the shift $\sigma_*$ on $\Sigma \cup \Sigma_*$ is a shift of finite type when $\mathcal{I}$ is finite and countable-state shift when $\mathcal{I}$ is countable. We use this embedding only as a notational device, allowing finite words to be treated as eventually constant infinite sequences. %To simplify the notation we are going to refer to $\sigma\colon \Sigma \cup \Sigma_* \to \Sigma \cup \Sigma_*$ as the shift in both spaces.

The thermodynamic formalism will be applied on the full shift $(\Sigma, \sigma)$, while estimates involving finite words will be obtained by restricting to the corresponding cylinders and the definition of the Ruelle operator.

%Unfortunately, the extension of $A$ is not irreducible, so $(\Sigma \cup \Sigma_*, \sigma_*)$ is not topologically mixing and therefore the thermodynamic formalism does not apply directly. However, we have a natural metric on $\Sigma \cup \Sigma_*$ as this shift of finite type. So we can define Hölder continuous functions in this space, Hölder continuous and try to use some result from Thermodynamic Formalism here to move it to this shift.

%In order to keep this notation, let $T_\emptyset\colon [0,1] \to [0,1]$ be the identity map. Then 
%$$\alpha_0 = \alpha_\emptyset = 1.$$

\subsection{Thermodynamic Formalism}

To state properly the main result, Theorem \ref{1MainTheo}, we recall some elements of the thermodynamic formalism.

Recall that $\Sigma$ is the space of all sequences taking values in the alphabet $\mathcal{I}$ and $\sigma\colon \Sigma \to \Sigma$ is the left shift.
%Let $\Sigma$ be the space of all sequences taking values in the alphabet $\mathcal{I}$, i.e.,
%$$\Sigma:= \mathcal{I}^\N = \{(x_1, x_2, \ldots) : x_j \in \mathcal{I}, \text{ for all }j\geq 1\}.$$
%We endow $\mathcal{I}$ with the discrete topology and $\Sigma$ with the product topology. Let $\sigma \colon \Sigma \to \Sigma$ be the left shift.

For $w \in \Sigma_*\setminus\emptyset$, we denote the $w$-cylinder
$$[w] := \{x\in \Sigma:  x_i = w_i, \,\forall i \in \{1,\ldots, |w|\}\}.$$

For a function $\gamma \colon \Sigma \to \R$, define the $n$-th Birkhoff sum as
$$S_n \gamma := \sum_{j = 0}^{n - 1} \gamma\circ \sigma^j \quad \text{ and } \quad S_0\gamma := 0. $$
The topological pressure function of $\gamma\colon \Sigma \to\R$ with respect to the shift map $\sigma\colon \Sigma \to \Sigma$, is defined by the well-defined limit
$$P(\gamma) :=\lim_{n\to\infty} \frac{1}{n} \log \sum_{w\in \Sigma_n} \exp \left(\sup_{\tau \in [w]} S_n \gamma(\tau)\right)$$
possibly with value $+ \infty$, for further details we refer to \cite[Chapter 2]{MauldinUrbanski}.

Let $\mathcal{C}(\Sigma)$ be the set of continuous real-valued functions on $\Sigma$. These functions are usually known as potentials. The set of bounded continuous functions in $\mathcal{C}(\Sigma)$ will be denoted by $\mathcal{C}_b(\Sigma)$. 

\begin{defi}[Hölder continuity]
For $\gamma \in \mathcal{C}(\Sigma)$ and $0 < \rho < 1$, let
\begin{align*}
    \text{var}_n(\gamma) &= \sup\{|\gamma(x)- \gamma(y)| : x_j = y_j, \,j= 1, \ldots, n\}\\
    \|\gamma\|_\rho &= \sup_{n\geq 1} \frac{\text{var}_n(\gamma)}{\rho^n},\\
    \mathcal{F}_\rho(\Sigma) &= \{\gamma\in \mathcal{C}(\Sigma): \|\gamma\|_\rho < \infty\}.
\end{align*}
\end{defi}

We denote by $\mathcal{F}_\rho^b(\Sigma)$ the subset of bounded Hölder continuous functions. By the observations made in the previous subsection, it is possible to equip $\Sigma \cup \Sigma_*$ with the subspace topology of the space $(\mathcal{I} \cup \{0\})^\N$. And then make these same definitions for $\mathcal{F}_\rho(\Sigma \cup \Sigma_*)$.

In order to study the central object of this section, namely the Perron-Frobenius operator of the potential $\gamma$, we will need to assume that
\begin{equation}\label{1.2SummableCondition}
    C_\gamma : = \sum_{e \in \mathcal{I}} \exp(\sup_{x\in [e]} \gamma(x)) < \infty.
\end{equation}
A function $\gamma$ satisfying \eqref{1.2SummableCondition} is called \textit{summable}.

\begin{defi}[Perron-Frobenius operator]
Let $\gamma\in \mathcal{F}_\rho(\Sigma)$ summable. The \textit{Perron-Frobenius operator} $\mathcal{L}_\gamma\colon \mathcal{C}_b(\Sigma) \to \mathcal{C}_b(\Sigma)$ for the potential $\gamma$ acting on $\mathcal{C}_b(\Sigma)$ is defined by
$$\mathcal{L}_\gamma(\varphi)(\underline{x}):= \sum_{\underline{y}: \sigma \underline{y} = \underline{x}} e^{\gamma(\underline{y})} \varphi(\underline{y}).$$
\end{defi}

The conjugate operator $\mathcal{L}_\gamma^*$ acting on $\mathcal{C}_b^* (\Sigma)$ can be restricted to the subset of finite Borel measures. 

\begin{teo}[{Real Ruelle-Perron-Frobenius Theorem for infinite alphabets, \cite[Theorem 2.4]{InfiniteRenewal}, see also \cite{MauldinUrbanski}}]\label{1RuellesTheorem}
Suppose that $\gamma \in \mathcal{F}_\rho(\Sigma)$ for some $\rho\in (0,1)$ is summable. Then $\mathcal{L}_\gamma$ preserves the space $\mathcal{F}_\rho^b(\Sigma)$. Moreover the following hold.
\begin{enumerate}[(i)]
\item There is a unique Borel probability eigenmeasure $\nu_\gamma$ of the conjugate Perron-Frobenius operator $\mathcal{L}_\gamma^*$ and the corresponding eigenvalue is equal to $e^{P(\gamma)}$. Moreover, $\nu_\gamma$ is a Gibbs measure for $\gamma$.

\item The operator $\mathcal{L}_{\gamma}|_{\mathcal{F}_\rho^b(\Sigma)}$ has an eigenfunction $h_\gamma$ which is bounded from above and satisfies $\int h_\gamma \,d\nu_\gamma = 1$. Furthermore, there exists $R> 0$ such that $h_\gamma \geq R$ on $\Sigma$.

\item The function $\gamma$ has a unique $\sigma$-invariant Gibbs measure $\mu_\gamma$.

\item There exist constants $\overline{M} > 0$ and $\theta \in (0,1)$ such that for every $\varphi \in \mathcal{F}_\rho^b(\Sigma)$ and $n\in \N$
\begin{equation}
    \left\| e^{-nP(\gamma)} \mathcal{L}^n_\gamma(\varphi)-  \int \varphi\,d\nu_\gamma \cdot h_\gamma \right\|_\rho \leq \overline{M} \theta^n (\|\varphi\|_\rho + \|\varphi\|_\infty).
\end{equation}
\end{enumerate}
\end{teo}

\begin{defi}\label{1.2RegularPotential}
A potential $\gamma\colon \Sigma \to \R$ is said to be a \textit{regular potential} if 
\begin{itemize}
    \item it is Hölder continuous, i.e. $\gamma\in \mathcal{F}(\Sigma)$ and
     $$ \int \gamma \,d\mu_{-\gamma} < \infty \quad \text{ and } \quad \inf \gamma > 0.$$
    \item the topological pressure of $-\gamma$ is zero;
    \item there exist $-1< t_* $ such that the potential $t\gamma$ is summable for every $t < t^*$.
\end{itemize}
\end{defi}

It is said that two functions $\gamma, \zeta\in \mathcal{F}_\rho(\Sigma)$ are \textit{cohomologous} if there exists $\varphi\in \mathcal{F}_\rho^b(\Sigma)$ such that 
$$\gamma - \zeta = \varphi - \varphi \circ \sigma.$$

If $\gamma, \zeta\in \mathcal{F}_\rho(\Sigma)$ are cohomologous for some $\rho \in (0,1)$ and they are both summable then they share the same invariant Gibbs measures. Furthermore if $\gamma - \zeta = \varphi - \varphi \circ \sigma$, for some $\varphi\in \mathcal{F}_\rho^b(\Sigma)$, then for every $\theta\in \mathcal{F}_\rho^b(\Sigma)$
$$\mathcal{L}_\gamma \theta = e^{-\varphi}\mathcal{L}_\zeta(e^{\varphi} \theta).$$
Consequently, assuming $P(\gamma) = P(\zeta) = 0$,
\begin{equation}\label{1.2EqCohomologus}
    \lambda_\gamma = \lambda_\zeta, \quad h_\gamma = \left( \int e^{\varphi} \,d\nu_\zeta\right) e^{-\varphi}h_\zeta, \quad \nu_\gamma = \frac{e^{\varphi} \nu_\zeta}{\int e^{\varphi}\,d\nu_\zeta}.
\end{equation}

\begin{defi}\label{1.2LatticeDefinition}
Let $\gamma\colon \Sigma\to \R$. We say that $\gamma$ is lattice if it is cohomologous to a function taking values in a discrete subgroup of the real numbers. Otherwise, $\gamma$ is called nonlattice.
\end{defi}

\subsection{Renewal Theory}

The key idea in the proof of Theorem \ref{1MainTheo} is to use a renewal theorem, adapted to symbolic dynamics.

Let $\gamma\in \mathcal{F}_\theta(\Sigma, \R)$ be a regular potential (recall Definition \ref{1.2RegularPotential}).

For $x\in \Sigma$ we will be interested in the asymptotic behaviour as $t\to \infty$ of the renewal function: 
\begin{equation}\label{2RenewalEquation}
    N(t,x) := \sum_{n = 0}^\infty \sum_{y: \sigma^n y = x} 1_{\{S_n\gamma(y) \leq t \}}(y),
\end{equation}
where, for each measurable set $A$, $1_A(x)$ is $1$ if $x\in A$ and $0$ otherwise.

It is called renewal because it satisfies the renewal equation
\begin{equation}\label{2RenewalEquationPaperLalley}
    N(t, x) = \sum_{ y : \sigma y = x} N(t - \gamma (y), y ) + 1\{t\geq 0\}.
\end{equation}

In order to use the renewal theorem for infinite alphabets we will need the following lemma, based \cite[Lemma 8.1]{Lalley}.

\begin{lem}\label{2Lema: LemaCuliaoConditinb}
Let $\gamma$ be a regular potential. There exists a constant $C < \infty$ depending only on $\gamma$ such that the renewal function $N(t,x)$ defined by \eqref{2RenewalEquation} satisfies
\begin{equation}
    N(t, x)\leq C e^t,
\end{equation}
for every $t\in \R$.
\end{lem}
\begin{proof}
We are going to prove this result for $\gamma$ satisfying that $S_n\gamma > c$ for some $n$ and certain $c > 0$ and $\inf_{0\leq j <n} S_j\gamma > -\infty$. Define 
$$G(t, x) = \frac{e^{-t} N(t, x)}{h_{-\gamma}(x)}.$$
It suffices to prove that $G(t,x) \leq C$ for some suitable constant.

The renewal equation \eqref{2RenewalEquationPaperLalley} may be rewritten as 
\begin{equation}
    G(t,x) = \sum_{y : \sigma^n y = x} G(t- S_n\gamma(y), y) e^{-S_n\gamma(y)} \cdot \frac{h_{-\gamma}(y)}{h_{-\gamma}(x)} + e^{-t} \sum_{j = 0}^{n - 1}\sum_{\sigma^j y=x} \frac{1_{\{t\geq S_j\gamma(y)\}}}{h_{-\gamma}(x)}.
\end{equation}
Since 
\begin{equation}\label{2Eq: hValorPropio}
    \sum_{y : \sigma^n y = x} e^{-S_n\gamma(y)} \frac{h_{-\gamma}(y)}{h_{-\gamma}(x)} = \frac{\left(\mathcal{L}^n h_{-\gamma}\right)(x)}{h_{-\gamma} (x)} = 1.
\end{equation}
Let 
$$\tilde{G}(t) = \sup_{r \leq t, x\in \Sigma} G(r,x).$$
First, we prove that $\tilde{G}(t) < \infty$.

Writing $N = qn + j$ 
$$S_N\gamma \geq qc + \underbrace{\inf_{0\leq j < n, x\in \Sigma} S_j\gamma(x)}_{\tilde c}.$$
Thus, recalling that 
$$G(t, x) = \frac{e^{-t}}{h_{-\gamma}(x)} \sum_{n= 0}^\infty \sum_{y: \sigma^n y = x} 1_{S_n\gamma(y)\leq t}(y),$$
then by hypothesis there exist $N > 0$ such that $S_n\gamma> t$ for all $n \geq N$, using the bound
$$1\{S_j \gamma(y) \leq t\} \leq e^{t - S_j\gamma(y)},$$
and the Ruelle operator $\mathcal{L}$, we get 
$$\tilde{G}(t) \leq \frac{N}{\inf h} < \infty.$$

Then, by \eqref{2Eq: hValorPropio} we have
$${G}(t, x) \leq \tilde{G}(t - c) + e^{-t} \sum_{j = 0}^{n - 1}\frac{ 1\{t\geq S_j\gamma(y)\}}{h_{-\gamma}(x)}.$$

Choose $s\in (0, 1)$ close to 1 such that $-s\gamma$ is summable. We clearly have 
$$1\{S_j \gamma(y) \leq t\} \leq e^{s(t - S_j\gamma(y))}.$$
Then, for each $j$ we have 
$$\sum_{\sigma^j y = x}\frac{ 1\{t\geq S_j\gamma(y)\}}{h_{-\gamma}(x)} \leq \frac{e^{st} \mathcal{L}^j_{-s\gamma}1(x)}{h_{-\gamma}(x)} \leq \frac{e^{st}\| \mathcal{L}_{-s\gamma}\|^j}{\inf h_{-\gamma}}.$$
Summing over all this $j < n$, we have
$$e^{-t} \sum_{j = 0}^{n - 1}\sum_{\sigma^j y = x}\frac{ 1\{t\geq S_j\gamma(y)\}}{h_{-\gamma}(x)} \leq e^{-t} \cdot \frac{e^{st}}{h_{-\gamma}(x)}\sum_{j = 0}^{n - 1} \mathcal{L}^j_{-s\gamma}1(x) \leq C_se^{-(1-s)t}.$$
Thus,
$$\tilde{G}(t) \leq \tilde{G}(t - c) + C_s e^{-(1-s)t}.$$
It is enough to conclude by iterating over intervals of length $c$ and summing a geometric series as in \cite[Lemma 8.1]{Lalley}.
\end{proof}

This result allows us to use the following Theorem.

\begin{teo}[{\cite[Theorem 3.1]{InfiniteRenewal}}]\label{2RTNonLatticeNormalRenewalTheorem}
Let $\gamma$ be a nonlattice regular potential in the sense of Definition \ref{1.2RegularPotential}. Then 
$$\lim_{t\to\infty}e^{-t} N(t,x) = \frac{h_{-\gamma}(x)}{\int \gamma\,d\mu_{-\gamma}} =: C_\gamma(x),$$
uniformly for $x\in \Sigma$.
\end{teo}

\begin{cor}\label{2Corollary: Cohomologus}
Let $\gamma'\colon \Sigma \to \R$ satisfy 
\begin{equation}\label{2Eq: Gamma'Coh}
    \gamma' = \gamma - u + u \circ \sigma,
\end{equation}
for some $u\in \mathcal{F}^b_\rho(\Sigma)$ and $\gamma$ a nonlattice regular potential in the sense of Definition \ref{1.2RegularPotential}. Let 
$$N'(t,x) := \sum_{n = 0}^\infty \sum_{y: \sigma^n y = x} 1_{\{S_n\gamma'(y) \leq t \}}(y).$$
Then 
$$\lim_{t\to\infty}e^{-t} N'(t,x) = \frac{h_{-\gamma'}(x)}{\int \gamma'\,d\mu_{-\gamma'}} =  \frac{e^{-u(x)}h_{-\gamma}(x) \int_\Sigma e^u \,d\nu_{-\gamma}}{\int \gamma\,d\mu_{-\gamma}} =: C_{\gamma'}(x),$$
uniformly for $x\in \Sigma$.
\end{cor}
\begin{proof}
Note that
$$\mathcal{L}_{z\gamma'} = M_{e^{zu}} \mathcal{L}_{z\gamma} M_{e^{-zu}},$$
where $M_{\bullet}$ is the operator multiplying by the function $\bullet$. Then 
$$(\text{Id} - \mathcal{L}_{z\gamma'})^{-1} = M_{e^{zu}}(\text{Id} - \mathcal{L}_{z\gamma})^{-1} M _{e^{-zu}},$$
and since $u$ is bounded we have the necessary information to deal with the poles of this function.

Equation \ref{2Eq: Gamma'Coh} gives us the nonlattice condition and
$$S_n \gamma' \geq nc - 2\|u\|_\infty \quad \text{ and }\quad N(t- 2\|u\|_\infty, x)\leq  N'(t, x)\leq N(t + 2\|u\|_\infty, x).$$
This also tells us that $N'(t,x) = 0$ for $t< -2\|u\|_\infty$.
Consequently, given that $N(t, x) \leq Ce^t$, we get 
$$N'(t,x) \leq Ce^{2\|u\|_\infty} e^t.$$
Finally by the cohomology condition we also have
$$P(-\gamma) = P(-\gamma') = 0,$$
and both equilibrium states are equal. 

With this, we can extend the Theorem using the same argument shown in \cite[Proof of Theorem 3.1 (i)]{InfiniteRenewal}.
\end{proof}

We will be especially interested in the case where we can extend this study to $\Sigma \cup \Sigma_*$. So first, we need to extend $N(\cdot, \cdot)$ to $\R \times (\Sigma \cup \Sigma_*)$.

Let $\gamma\colon \Sigma \cup \Sigma_* \to \R$. For $x\in \Sigma$, set $N_*(t,x) = N(t,x)$. For $x\in \Sigma_*$ and $t\in \R$ define 
$$N_*(t,x) := 1 \{0 \leq t\} + \sum_{n = 1}^\infty \sum_{\substack{\sigma_*^n y = x\\\sigma_*^{n - 1}y \neq \emptyset }} 1\{S_n \gamma(y) \leq t\}.$$

\begin{teo}\label{2RenewalTheoremNLFiniteWords}
Let $\gamma\colon \Sigma \cup \Sigma_* \to \R$ be such that $\gamma\in \mathcal{F}_\rho(\Sigma \cup \Sigma_*)$ and $\gamma|_\Sigma$ is a regular nonlattice potential or $\gamma|_\Sigma$ satisfies the hypotheses of Corollary \ref{2Corollary: Cohomologus}. Then 
\begin{equation}
    \lim_{t \to \infty} e^{-t} N_*(t,x) = \frac{h_*(x)}{ \int \gamma \,d\mu_{-\gamma}} = C_\gamma(x),
\end{equation}
where $h_*$ is the unique function in $\mathcal{F}_\rho^b(\Sigma\cup \Sigma_*)$ satisfying $h_*|_\Sigma = h_{-\gamma}$ and for $x\in \Sigma_*$
\begin{equation}\label{2CondicionhEstrella}
    h_*(x) := \sum_{\substack{y : \sigma_* y= x\\ y \neq \emptyset}} e^{-\gamma(y)} h_*(y)
\end{equation}
\end{teo}

For the finite alphabet case, this theorem is just Theorem 4 in \cite{Lalley}. We will follow the same idea, adapted to an infinite alphabet. For the proof we will need the following results adapted to the infinite alphabet case

\begin{lem}\label{2Lem: UniquenessOfH}
There is at most one nonnegative $h_*(x)\in \mathcal{F}_\rho^b(\Sigma \cup \Sigma_*)$ satisfying $h_*|_{\Sigma}= h_{-\gamma}$ and \eqref{2CondicionhEstrella}.
\end{lem}
\begin{proof}
Note that any $h_*$ satisfying $h_*|_{\Sigma} = h_{-\gamma}$ and \eqref{2CondicionhEstrella} must be strictly positive.

For $v\in \Sigma_*$ and $\underline{v}\in [v]$. Then 
$$|h_*(v) - h_{-\gamma}({\underline{v}})|\leq \rho^{|v|} |h_*|_\rho.$$
Thus, $h_*(v) \geq R/2$ for all sufficiently long finite words $v$. Iterating \eqref{2CondicionhEstrella}, we get that $h_*$ is strictly positive on $\Sigma_*$.

For $x,y\in \Sigma_*$ and $y \neq \emptyset$, define 
$$k(x,y) := \begin{cases}
    e^{-\gamma(y)} \cdot (h_*(y) / h_*(x)) &\text{ if } \sigma_* y = x,\\
    0, &\text{otherwise.}
\end{cases}$$
Then for each $x\in \Sigma_*$, we have $\sum_{y: \sigma_* y = x} k(x,y) = 1$. Define 
$$k^n(x, y) = \sum_{(y_1, \ldots, y_{n - 1})} k(x, y_1)k(y_1, y_2)\cdots k(y_{n - 1}, y).$$
By induction, $\sum_{y :\sigma_*^n y=  x} k^n(x,y) = 1$. Now suppose that $h'_*$ is another nonnegative continuous function satisfying \eqref{2CondicionhEstrella} and $h_*' |_{\Sigma} = h_{-\gamma}$. Then 
$$\frac{h_*'(x)}{h_*(x)} = \sum_{y: \sigma_* y = x} k(x,y) \left(\frac{h_*'(y)}{h_*(y)} \right) = \sum_{y: \sigma_*^n y = x} k^n(x,y) \left(\frac{h_*'(y)}{h_*(y)} \right).$$
Since $h_*, h_*' \in \mathcal{F}_\rho^b(\Sigma \cup \Sigma_*)$ and both restrict to $h_{-\gamma}$ on $\Sigma$. We have the following estimate. For $v\in \Sigma_*$ and $\underline{v}\in [v]$. Then 
$$|h_*(v) - h_{-\gamma}({\underline{v}})|\leq \rho^{|v|} |h_*|_\rho.$$
Therefore 
$$|h_*'(v) - h_*(v)| \leq (|h_*|_\rho + |h_*'|_\rho) \rho^{|v|}.$$
Hence 
$$\sup_{\substack{v\in\Sigma_*\\ |v|\geq m}}
\left|
\frac{h_*'(v)}{h_*(v)}-1
\right|
\longrightarrow 0
\quad\text{as }m\to\infty.$$

Finally, fix $x\in \Sigma_*$. If $\sigma_*^n y = x$, then 
$$\frac{h_*'(y)}{h_*(y)} \to 1,$$
uniformly over all $y$ with $\sigma_*^n y = x$. Then 
$$\frac{h_*'(x)}{h_*(x)} = \sum_{\sigma_*^n y = x} k^n(x,y) \frac{h_*'(y)}{h_*(y)}  \to 1.$$
Since $x$ is fixed, this would imply that the quotient is $1$, and thus, both functions are equal.

Since $x$ was arbitrary, the uniqueness of $h_*$ follows.  
\end{proof}

\begin{lem}\label{2RenewalLemmaEpsilon}
For each $\varepsilon > 0$ there exists $n_\varepsilon$ sufficiently large that if $x,x' \in \Sigma \cup \Sigma_*$ satisfy $x_i = x_i'$ for $i = 1, \ldots, n_\varepsilon$, then
\begin{equation}
    N_*(t,x) \leq N_*(t + \varepsilon, x'),
\end{equation}
for all $t\in \R$.
\end{lem}
\begin{proof}
Since $\gamma$ is Hölder, its variations are summable. Take $n_\varepsilon$ big enough so that 
$$\sum_{m \geq n_\varepsilon} \text{var}_m(\gamma) < \varepsilon.$$
Suppose that $x_i = x_i '$ for $i = 1, \ldots, n_\varepsilon$. If $y$ is an $n$-th preimage of $x$, and $y'$ an $n$-th preimage of $x'$ with the same first $n$ coordinates equal, then 
$$|S_n \gamma(y) - S_n\gamma(y') | \leq \sum_{m = n_\varepsilon}^\infty \text{var}_m (\gamma) \leq \varepsilon.$$
Therefore 
$$S_n\gamma(y) \leq t \implies S_n\gamma(y') \leq t + \varepsilon.$$
Then, summing over all $n$ and over all $n$-preimages, we obtain 
$$N_*(t, x) \leq N_*(t + \varepsilon, x').$$
\end{proof}
\begin{proof}[Proof of Theorem \ref{2RenewalTheoremNLFiniteWords}]
Iterating the renewal equation $n$ times, we get
\begin{equation}\label{2RenewalExpanded}
    N_*(t,x) = \sum_{\substack{y: \sigma_*^n y = x\\ \sigma_*^{n - 1}y \neq \emptyset}} N_*(t - S_n \gamma(y), y) + \sum_{j = 1}^{n - 1} \sum_{\substack{y : \sigma_*^j y = x\\ \sigma_*^{j - 1} y\neq \emptyset}} 1\{t - S_j\gamma(y) \geq 0\} + 1\{t \geq 0\}.
\end{equation}
Note that the summability condition ensures that both terms are finite. We first show that, for fixed $n$,
\begin{equation}\label{2RTRestosA0Fixing}
    \lim_{t\to\infty}e^{-t}\left(\sum_{j = 1}^{n - 1} \sum_{\substack{y : \sigma_*^j y = x\\ \sigma_*^{j - 1} y\neq \emptyset}} 1\{t - S_j\gamma(y) \geq 0\} + 1\{t \geq 0\}\right) = 0,
\end{equation}
as $t\to \infty$.
Fix $j\in \{1, \ldots, n - 1\}$. We have that 
$$e^{-t}\sum_{\substack{y : \sigma_*^j y = x\\ \sigma_*^{j - 1} y\neq \emptyset}} 1\{t - S_j\gamma(y) \geq 0\} = \sum_{\substack{y : \sigma_*^j y = x\\ \sigma_*^{j - 1} y\neq \emptyset}} e^{-t}1\{t - S_j\gamma(y) \geq 0\}.$$
Note that 
$$e^{-t}1\{t - S_j\gamma(y) \geq 0\} \leq e^{-S_j \gamma(y)},$$
by the summability condition, we get that 
$$\sum_{\substack{y : \sigma_*^j y = x\\ \sigma_*^{j - 1} y\neq \emptyset}}e^{-S_j \gamma(y)} < \infty.$$
Then, by Dominated Convergence Theorem , we get
$$\lim_{t\to\infty} e^{-t}\sum_{\substack{y : \sigma^j y = x\\ \sigma^{j - 1} y\neq \emptyset}} 1\{t - S_j\gamma(y) \geq 0\} = 0.$$
This proves \eqref{2RTRestosA0Fixing}. Thus the asymptotic behaviour of $N_*(t, x)$ is completely determined by 
$$\sum_{\substack{y: \sigma_*^n y = x\\ \sigma_*^{n - 1}y \neq \emptyset}} N_*(t - S_n \gamma(y), y).$$
Each $y$ in the preceding expression is a sequence of length at least $n$. For each $y$ there exists $y' \in \Sigma$ such that $y_i ' = y_i$ for $i = 0, 1, \ldots, n- 1$. By Lemma \ref{2RenewalLemmaEpsilon}, if $n\geq n_\varepsilon$ then 
\begin{equation}\label{2Eq: ConsequenceLemmaBound}
    N(t- \varepsilon - S_n \gamma(y), y') \leq N_*(t - S_n\gamma(y), y) \leq N(t + \varepsilon - S_n\gamma(y), y').
\end{equation}
Since Theorem \ref{2RTNonLatticeNormalRenewalTheorem} (or Corollary \ref{2Corollary: Cohomologus} for more generality) holds uniformly for $y'\in \Sigma$, using the summability condition on $\gamma$ and Lemma \ref{2Lema: LemaCuliaoConditinb}, to use domintaed convergence theorem, we get
$$N(t \pm \varepsilon - S_n\gamma(y), y') \sim C(y') e^{t - S_n\gamma(y)} e^{\pm \varepsilon},$$
as $t\to\infty $. 
Thus,
\begin{align*}
\sum_{\substack{y : \sigma^n_* y = x\\ 
\sigma^{n - 1} y\neq \emptyset}} C(y')\cdot e^{-S_n\gamma(y) - \varepsilon} &\leq \liminf_{t\to\infty} e^{-t} N_*(t,x)\\
&\leq \limsup_{t \to \infty} e^{-t}N_*(t, x) \leq \sum_{\substack{y : \sigma^n_* y = x\\ 
\sigma^{n - 1} y\neq \emptyset}} C(y')\cdot e^{-S_n\gamma(y) + \varepsilon}.
\end{align*}
In other words,
$$\limsup_{t\to\infty} e^{-t} N_*(t, x) \leq e^{2\varepsilon}\liminf_{t\to\infty}e^{-t} N_*(t, x).$$
Taking $\varepsilon\to 0$, we conclude that
$$\limsup_{t\to \infty} e^{-t} N_*(t, x) = \liminf_{t \to\infty} e^{-t} N_*(t, x).$$
Defining 
$$h_*(x) = \int_\Sigma \gamma\, d\mu_{-\gamma} \cdot \lim_{t\to\infty} e^{-t} N_*(t, x).$$
After taking the renewal limits, we get that \eqref{2Eq: ConsequenceLemmaBound} gives that for 
$$D_n := \sum_{j = n+ 1}^\infty\text{var}_j(\gamma) \leq \frac{\|\gamma\|_\rho \rho^{n+ 1}}{1 - \rho}.$$
Then the proof of Lemma \ref{2RenewalLemmaEpsilon} gives
$$N_*(t- D_n, y) \leq N_*(t, x) \leq N_*(t + D_n, y),$$
whenever $x$ and $y$ agrees in their first $n$ symbols.
With this we get
$$e^{-D_n} h_*(y) \leq h_*(x)\leq e^{D_n}h_*(y)$$
for all $x\in \Sigma \cup \Sigma_*$. Taking $n = 0$ we obtain
$$0 < h_*(x) \leq e^{D_0} \| h_{-\gamma}\|_\infty.$$
And 
$$|h_*(x) - h_*(y) | \leq \|h_*\|_\infty \left(e^{D_n} - 1\right) \leq C \rho^n,$$
for some $\rho \in (0,1)$. Dominated convergence theorem in the renewal equation for $N_*$ gives 
$$\sum_{\substack{\sigma_*^n y = x\\ y\neq \emptyset}} e^{-S_n\gamma}(y) h_{*}(y) = h_*(x).$$
Since $h_*|_{\Sigma} = h_{-\gamma}$, we conclude the uniqueness by Lemma \ref{2Lem: UniquenessOfH}.

\end{proof}

%\section{Renewal Theory}
%\input{2RenewalTheory}

\section{Proof of Main Theorem}\label{SectionProof}
\subsection{Reduction to cylinders}

Let $\mathcal{P}$ be the partition generated by the transformations $\{T_i\}_{i\in \mathcal{I}}$, where $\mathcal{I}$ is at most countable and the transformations satisfy the hypotheses of Theorem \ref{1MainTheo}.

As we mentioned before, we are going to follow the approach of Pollicott and Sewell in \cite{PollicottSewell}. It will be enough to prove the following Lemma.

\begin{lem}\label{2MainLem}
For all words $v\in \Sigma_*$,   
$$\lim_{\lambda \to 0^+} \mu_\lambda (T_v[0,1)) = \leb(T_v[0,1)).$$
\end{lem}

We will prove this lemma at the end of the section.

\begin{proof}[Proof of Theorem \ref{1MainTheo}]
Let $c_M \in (0,1)$ be a fixed constant such that for all $i\in \mathcal{I}$ and all $x\in (0,1)$ we have
$$T_i'(x) < c_M < 1.$$
Let $v\in \Sigma_*$ be an arbitrary nonempty word. By the mean value theorem, we have that 
$$\leb(T_v([0,1])) = T_v(1) - T_v(0) = T_v'(y),$$
for some $y\in [0,1]$. Using the chain rule, we get that 
$$\leb(T_v([0,1])) \leq c_M^{|v|}.$$

Now the proof works like the one in \cite{PollicottSewell}. Let $I\subset [0,1]$ be an interval. Fix $k \in \N$, let 
$$\mathcal{U}_k:= \{T_v[0,1):  v\in \Sigma_k, \, T_v[0,1) \subset I\}.$$
Since $\{T_i[0,1)\}$ are pairwise disjoint, it follows inductively that the intervals in $\mathcal{U}_k$ are pairwise disjoint and contained in $I$. Therefore,
$$\sum_{U \in \mathcal{U}_k} \mu_{\lambda}(U) \leq \mu_\lambda(I).$$
Moreover, by Lemma \ref{2MainLem}, for every $U \in \mathcal{U}_k$, 
$$\lim_{\lambda\to 0^+} \mu_\lambda(U) = \leb(U).$$
Therefore, by Fatou's Lemma 
$$\liminf_{\lambda\to0^+} \mu_\lambda(I) \geq \sum_{U \in \mathcal{U}_k} \leb(U) \geq \leb(I) - 2c_M^k.$$

We now show that 
\begin{equation}\label{5ProofOfLebesgue}
    \sum_{U\in \mathcal{U}_k} \leb(U) \geq \leb(I) - 2c_M^k.
\end{equation}

Let $x\in I\setminus \bigcup_{U\in \mathcal{U}_k}U$. Then one of the following cases holds:
\begin{itemize}
    \item[Case 1:] $x\in K_k:= [0,1] \setminus \bigcup_{v\in \Sigma_k} T_v[0,1)$. But we know that $\leb(K_k) = 0$ for all $k$.
    \item[Case 2:] there exists $v\in \Sigma_k$ such that $x\in T_v[0,1) \nsubseteq I$. So $T_v[0,1)$ must meet an endpoint of $I$. Hence, there are at most two $v\in \Sigma_k$ with this property, and $x$ is contained in the union of at most two intervals, each with length at most $c_M^k$. This proves \eqref{5ProofOfLebesgue}.
\end{itemize}

Since $\mu_{\lambda}(U) \to \leb(U)$, for each $U \in \mathcal{U}_k$, then by Fatou's Lemma 
$$\liminf_{\lambda \to 0^+} \mu_\lambda(I) \geq \lim_{\lambda\to 0^+} \sum_{U \in \mathcal{U}_k} \mu_\lambda(U) = \sum_{U\in \mathcal{U}_k} \leb(U) \geq \leb(I) - 2c_{M}^k.$$

Applying the same argument to the connected components of $[0,1] \setminus I$ gives 
$$\liminf_{\lambda\to0^+}\mu_\lambda([0,1]\setminus I) \geq \leb([0,1]\setminus I)-4c_M^k.$$
Since $(\mu_\lambda)$ and Lebesgue measure are probability measures, this implies
$$\limsup_{\lambda \to 0^+} \mu_\lambda(I) \leq \leb(I) + 4c_M^k.$$
Taking $k\to\infty$, we finish the proof of the equidistribution of the sequence $\{L_n\}_n$.
\end{proof}

\subsection{Endpoints and words}
The objective of this subsection is to prove Lemma \ref{2MainLem}. Throughout the subsection, we fix a nonempty word $v \in \Sigma_*$. We first relate the function $\# X_\lambda$ to $\#A_\lambda$ and then obtain an analogous relation between the set $X_\lambda \cap T_v[0,1)$ and a specific subset of $A_\lambda$.

The following result follows from \cite[Lemma 1]{PollicottSewell}. %After relabelling the alphabet if necessary, let $1\in \mathcal{I}$ be the unique symbol such that $T_1(0) = 0$.

\begin{lem}\label{2XLambdaCardinality}
For every $\lambda > 0$, either there is a unique symbol $1\in \mathcal{I}$ such that $T_1(0) = 0$, in which case
\begin{equation}
\#X_\lambda = \# A_\lambda - \#\{u\in A_\lambda: u1\in A_\lambda\};  
\end{equation}
or no element of $\{T_v\}_{v\in \Sigma_*\setminus \{\emptyset\}}$ fixes $0$, and $\#X_\lambda = \# A_\lambda$.
\end{lem}
\begin{proof}
First, suppose there is a symbol $1\in \mathcal{I}$ such that $T_1 (0) = 0$. Define
$$R_\lambda  = A_\lambda\setminus \{v1: v,v1\in A_\lambda \}.$$
We claim that the map $R_\lambda \to X_\lambda$ given by $u\mapsto T_u(0)$, is a bijection.

Let $u,w \in A_\lambda$ be such that $T_u (0) = T_w(0)$. By disjointness of the half-open intervals $T_i[0,1)$, one sees that one of the words must be an extension of the other. Without loss of generality, assume $w$ is an extension of $u$. Then, there exists $z\in \Sigma_*$ such that $w = uz$. Since $T_u$ is injective, this implies that $T_z(0) = 0$. By the uniqueness of the symbol 1, it follows by induction that $z = 1^j$ for some $j \geq 0$.

Thus, two words in $A_\lambda$ define the same point if and only if they differ by appending a finite string of $1$'s. Since every equivalence class has a unique representative $u\in A_\lambda$ such that there is no word $u'$ with $|u'| < |u|$ and $T_{u'}(0) = T_u(0)$. These are exactly the elements of $R_\lambda$. Therefore the map $ u\mapsto T_u(0)$ is a bijection from $R_\lambda \to X_\lambda$. Hence 
$$\#X_\lambda = \#R_\lambda = \#A_\lambda - \#\{u \in A_\lambda: u1 \in A_\lambda\}.$$

On the other hand, if no word fixes zero, then the map $v\mapsto T_v(0)$ is clearly a bijection. Thus $\#X_\lambda = \# A_\lambda$.
\end{proof}

Let 
$$A^v_\lambda = \{u: vu \in A_\lambda\}.$$

\begin{lem}\label{3.2MedidaCilindro}
Let $v\in \Sigma_*$ and $0 < \lambda \leq \alpha_v$. Either there is a unique symbol $1\in \mathcal{I}$ such that $T_1(0) = 0$, in which case
\begin{equation}
    \#(X_{\lambda} \cap T_v[0,1)) =\# A^v_\lambda - \#\{u: u\in A^v_\lambda \text{ and } u1\in A^v_\lambda\}.
\end{equation}
Equivalently,
$$\#(X_{\lambda} \cap T_v[0,1)) =\# A^v_\lambda - \#\{u: u\in A^v_\lambda \text{ and } vu1\in A_\lambda\};$$
or no element of $\{T_v\}_{v\in \Sigma_*\setminus \{\emptyset\}}$ fixes $0$, and 
$$\#(X_{\lambda} \cap T_v[0,1)) =\# A^v_\lambda .$$
\end{lem}
\begin{proof}

Assume that there is a unique symbol $1\in \mathcal{I}$ such that $T_1(0) = 0$. Since $0 < \lambda \leq \alpha_v$, we have $v\in A_\lambda$. Therefore $\emptyset \in A_\lambda^v$. Define 
$$R^v_\lambda := A^v_\lambda \setminus \{u \in A^v_\lambda: u1 \in A^v_\lambda\}.$$
We claim that the map $R_\lambda^v \to X_\lambda \cap T_v[0,1)$ defined by $u\mapsto T_{vu}(0)$ is a bijection.

First, if $u\in R^v_\lambda$, then $vu\in A_\lambda$, so $T_{vu}(0)\in X_\lambda$. Moreover, 
$$T_{vu}(0) \in T_v[0,1),$$
so the map is well defined.

To prove surjectivity, let $z\in X_\lambda \cap T_v[0,1)$. Then $z = T_w(0)$ for some $w \in A_\lambda$. Since $z\in T_v[0,1)$, either $w = vu$ or $v = wu$ for some $u\in \Sigma_*$. If $w = vu$, $w \in A_\lambda$, we have that $z$ is the image of $u$ for $u\in A_\lambda^v$. If $v = wu$, inductively we can see that $v= w1^k$, for some $k\geq 0$. So $z = T_v(0)$ which is the image of $\emptyset$.  Therefore every point in $X_\lambda \cap T_v[0,1)$ has a representative of the form $T_{vu}(0)$ with $u\in A_\lambda^v$. Replacing $u$ by its representative $u1^j$, we get an element of $R^v_\lambda$.

Injectivity follows as in Lemma \ref{2XLambdaCardinality}. Indeed, if $u,w \in R^v_\lambda$ and $T_{vu}(0) = T_{vw}(0)$, then $T_u(0) = T_w(0)$. Hence $u$ is obtained by $w$ or vice versa by appending a finite string of $1$'s. Since both are in $R_\lambda^v$, this implies that they must be equal.

Thus the map is a bijection and we conclude that 
$$\#(X_\lambda \cap T_v[0,1)) = \# R^v_\lambda  = \#A_\lambda^v - \#\{u \in A^v _\lambda : u1\in A^v_\lambda\}.$$

On the other hand, if no symbol fixes zero, we get that the map $u\mapsto T_{vu}(0)$ is a bijection.
\end{proof}

\subsection{Renewal estimates}

The objective of this section is to establish a renewal equation for the amounts $\#A_\lambda$ and $\#A^v_\lambda$. Once we have done this, it will be possible to have an expression for $\mu_\lambda(T_v[0,1))$ in terms of this equation, and then compute the limit as $\lambda\to 0^+$

Let us consider the following map $\pi \colon \Sigma \to [0,1]$
$$\pi(\underline{x}) = \lim_{n \to\infty} T_{x_1\dots x_n}(0).$$
Equivalently,
$$\{\pi(\underline{x})\} = \bigcap_{n = 1}^\infty T_{x_1\dots x_n}([0,1]).$$

Let $f\colon [0,1]\to [0,1]$ be the full branch expanding map described in \eqref{1.1DefiSistemaDinamico0,1}.

We will prove that $\pi$ is continuous. Consider $\underline{x} = ( x_1, x_2 \ldots)\in \Sigma$ and for each $n$ define 
$$I_{x_1, \ldots, x_n} = \{x \in [0,1]: x\in T_{x_1}([0,1]), \,f(x) \in T_{x_2}([0,1]), \ldots, f^{n - 1}(x) \in T_{x_n}([0, 1])\}.$$
Let 
$$c_{n }(\underline{x}) = \{\underline{y} \in \Sigma: x_1 = y_1, \ldots, x_n = y_n\}.$$
Then 
$$\pi(c_{n }(\underline{x})) \subseteq I_{x_1, \ldots, x_n}.$$
Since $\text{diam}(I_{x_1,\ldots, x_n}) \leq c_M^n$, we get that $\pi$ is continuous.

We also know that the following diagram commutes:
$$% https://tikzcd.yichuanshen.de/#N4Igdg9gJgpgziAXAbVABwnAlgFyxMJZABgBpiBdUkANwEMAbAVxiRAB12BlLAcwFs6IAL6l0mXPkIoAjOSq1GLNpx4Cho8djwEiZGQvrNWiECVIGRYkBm1Sicg9SPLT5y8IUwoveEVAAZgBOEPxIZCA4EEhyisYq7NjqVoEhYYgATNRRSADMzkomIAEpxWnh2dGZBfGmnGhYpcGhMZV5Na4c7A0iFMJAA
\begin{tikzcd}
\Sigma \arrow[r, "\sigma"] \arrow[d, "\pi"] & \Sigma \arrow[d, "\pi"] \\
{[0,1]} \arrow[r, "f"]                      & {[0,1]}                
\end{tikzcd}$$

Let $\beta\colon \Sigma\cup \Sigma_* \to \R$ defined by
$$\beta(\emptyset) = 1, \quad \beta(x_1) = \alpha_{x_1},\quad \beta(x_1\cdots x_n) = \frac{\alpha_{x_1\dots x_n}}{\alpha_{x_2\dots x_n}} $$
for finite words and for $\Sigma$
$$\beta(x) = \lim_{n\to\infty } \frac{\alpha_{x_1\dots x_n}}{\alpha_{x_2\dots x_n}}.$$

By the mean value theorem
$$\frac{\alpha_{x_1\dots x_n}}{\alpha_{x_2\dots x_n}} = \frac{T_{x_1}(T_{x_2\ldots x_n} (1)) - T_{x_1}(T_{x_2\ldots x_n}(0))}{T_{x_2\ldots x_n}(1) -T_{x_2\ldots x_n}(0)} = T'_{x_1}(y_n),$$
for some $y_n \in [T_{x_2\ldots x_n}(0), T_{x_2\ldots x_n}(1)]$. Then, 
$$\lim_{n\to \infty}y_n = \pi(\sigma\underline{x}).$$
Consequently, by the continuity of $T_i'$ and the inverse function theorem,
\begin{equation}\label{2DefBeta}
    \beta(\underline{x}) =  T_{x_1}'(\pi(\sigma \underline{x})) =\frac{1}{f'(T_{x_1}(\pi(\sigma \underline{x})))}  =   \frac{1}{f'(\pi(\underline{x} ))} ,
\end{equation}
for all $\underline{x}\in \Sigma$. As a consequence $\beta|_\Sigma \in \mathcal{F}_\rho(\Sigma)$ for some $\rho \in (0,1)$.

Let $\gamma\colon \Sigma\cup \Sigma_* \to \R$ defined by 
\begin{equation}\label{2DefGamma}
    \gamma(\underline{x}) = - \log \beta(\underline{x}).
\end{equation}
In particular, for $\underline{x}\in \Sigma$, we have 
$$\gamma(\underline{x}) = \log f'(\pi( \underline{x} )).$$

Furthermore, by our assumptions on $\{T_i\}_{i \in \mathcal{I}}$ and by the Mean Value Theorem, for all $x\in \Sigma \cup \Sigma_* \setminus \{\emptyset\}$ we have $\gamma(x)$ is strictly greater than a constant $c > 0$. Since $\gamma = -\log \beta$ and $\beta|_{\Sigma} $ is bounded away from $1$, we have $\gamma|_\Sigma\in \mathcal{F}_\rho(\Sigma)$ and it is bounded away from zero.

Let ${N}\colon \R\times (\Sigma \cup \Sigma_*) \to \R$, given by 
$${N}(t, \underline{x}) = \sum_{n = 0}^{\infty}\sum_{\substack{\underline{y}: \sigma_*^n \underline{y} = \underline{x} \\ \sigma_*^j \underline{y} \neq \emptyset, \, j<n}}1\{S_n\gamma(\underline{y}) \leq t\}.$$

Because $S_n \gamma$ is strictly positive on $\Sigma$ and $\gamma$ is a regular potential, ${N}(t, \underline{x})$ is finite for all $(t, \underline{x})\in \R\times \Sigma$. This function also satisfies the following renewal equation
\begin{equation}\label{3.3RenewalEquation}
    {N}(t, \underline{x}) = \sum_{\substack{\underline{y}: \sigma_* \underline{y} = \underline{x}\\ y\neq \emptyset}} {N}(t - \gamma(\underline{y}), \underline{y}) + 1\{t\geq 0\}.
\end{equation}

For $\lambda >0$ let
$$\widetilde{N}(\lambda, \underline{x}) = {N}(-\log \lambda, \underline{x}).$$

We will establish a relation between this function and $\# A_\lambda$.

\begin{lem}\label{2Almabda}
For all $\lambda \in (0,1]$, we have 
$$\# A_\lambda = \widetilde{N}(\lambda, \emptyset).$$
\end{lem}
\begin{proof}
Note that % Tenemos que incluir al vacio siempre.
\begin{align*}
    A_\lambda &= \bigcup_{n = 0}^\infty\{x_1\cdots x_n \in \Sigma_*: \alpha_{x_1\cdots x_n} \geq \lambda\}\\
    &= \bigcup_{n = 0}^\infty\left\{x_1\cdots x_n \in \Sigma_*: \frac{\alpha_{x_1\cdots x_n}}{\alpha_{x_2\cdots x_n}}\cdot \frac{\alpha_{x_2\cdots x_n}}{\alpha_{x_3\cdots x_n}}\cdots \frac{\alpha_{x_{n - 1}x_n}}{\alpha_{x_n}} \cdot \alpha_{x_n} \geq \lambda\right\}.
\end{align*}
Using the definition of the function $\beta$ and the fact that a word $v$ is of length $n$ if and only if $v\in \sigma_*^{-n}\emptyset$ and $\sigma_*^j v \neq \emptyset$ for $j < n$, we have
$$A_\lambda = \bigcup_{n = 0}^{\infty}\bigcup_{\substack{\substack{y: \sigma_*^n y = \emptyset \\ \sigma_* ^jy\neq \emptyset\\ j< n}}}\left\{y: \beta(y) \cdot \beta(\sigma_* y) \cdots  \beta(\sigma_*^{n - 1} y) \geq \lambda \right\}.$$
Since these unions are disjoint, this implies that 
$$\#A_\lambda = \sum_{n = 0}^{\infty} \sum_{\substack{\substack{y: \sigma_*^n y = \emptyset \\ \sigma_* ^jy\neq \emptyset\\ j< n}}}1\left\{ \beta(y) \cdot \beta(\sigma_* y) \cdots  \beta(\sigma_*^{n - 1} y) \geq \lambda \right\}.$$

Finally, by the definition of $\gamma$, we have 
\begin{equation}\label{2EquivLogConBeta}
    \beta(y) \cdot\beta(\sigma_*y)\cdots \beta(\sigma^{n - 1}_*y) \geq \lambda\iff S_n\gamma(y)\leq - \log\lambda.
\end{equation}
Therefore,
$$\# A_\lambda = \sum_{n = 0}^\infty \sum_{\substack{y: \sigma_*^n y = \emptyset\\ \sigma_* ^jy\neq \emptyset\\ j< n}} 1\{S_n \gamma(y) \leq -\log \lambda\} = \widetilde{N}(\lambda, \emptyset).$$
\end{proof}

\begin{cor}
Assume there is a symbol $1 \in \mathcal{I}$ such that $T_1(0) = 0$. Then, for all $\lambda\in (0,1]$
\begin{equation}\label{3.3EquationforXlambda}
\#X_\lambda = \widetilde{N}(\lambda, \emptyset) - \widetilde{N}(\lambda/\alpha_1, 1).   
\end{equation}
\end{cor}

\begin{proof}
By Lemma \ref{2XLambdaCardinality}, we have 
$$\# X_\lambda = \#A_\lambda - \# \{v\in A_\lambda: v1\in A_\lambda\}.$$
By Lemma \ref{2Almabda}, 
$$\#A_\lambda = \widetilde{N}(\lambda, \emptyset).$$
It remains to prove that
$$\#\{v\in A_\lambda: v1\in A_\lambda\} = \widetilde{N}(\lambda/\alpha_1, 1).$$
Indeed
\begin{align*}
    \# \{v:v1\in A_\lambda\} &= \# \bigcup_{n = 1}^{\infty}\bigcup_{\substack{\substack{y: \sigma_*^{n- 1} y = 1}}}\left\{y: \beta(y) \cdot \beta(\sigma_* y) \cdots \beta(\sigma_*^{n - 2}y)\cdot \beta(1) \geq \lambda \right\}\\
    &= \sum_{n = 0}^{\infty} \sum_{\substack{\substack{y: \sigma_*^{n} y = 1}}}1\left\{ \beta(y) \cdot \beta(\sigma_* y) \cdots  \beta(\sigma_*^{n - 1} y) \geq \lambda/\alpha_1 \right\}.
\end{align*}
By equation \eqref{2EquivLogConBeta}, we conclude 
\begin{align*}
    \# \{v:v1\in A_\lambda\} &= \sum_{n = 0}^{\infty} \sum_{\substack{\substack{y: \sigma_*^{n} y = 1}}}1\left\{ \gamma(y) + \gamma(\sigma_*y) + \ldots + \gamma(\sigma_*^{n - 1} y) \leq - \log{(\lambda/\alpha_1)} \right\}\\
    &= \widetilde{N}(\lambda/\alpha_1, 1).
\end{align*}
\end{proof}

Now, we will repeat the previous construction of the Renewal Equation to get another Renewal equation for $\#A_v^\lambda$. Let $\beta_v\colon \Sigma\cup \Sigma_*\to \R$ be a function defined by $\beta_v(\emptyset)= 1$, for a nonempty finite word $x_1\ldots x_n$, 
$$\beta_v(x_1\ldots x_n) = \frac{\alpha_{vx_1\ldots x_n}}{\alpha_{vx_2\ldots x_n}},$$
and for $\underline{x}\in \Sigma$:
\begin{equation}
    \beta_v(\underline{x}): = \lim_{n\to\infty} \frac{\alpha_{vx_1\ldots x_n}}{\alpha_{vx_2\ldots x_n}}.
\end{equation}
%In any case, for $x\in \Sigma \cup \Sigma_*$:
%$$\beta_v(x) =  \frac{\beta(v\underline{x}) \cdot \beta(\sigma v \underline{x}) \cdots \beta(\sigma^{|v| - 1}v \underline{x})}{\beta(v\sigma \underline{x}) \cdot \beta(\sigma v\sigma \underline{x}) \cdots \beta(\sigma ^{|v| - 1}v\sigma \underline{x})} \cdot \beta(\underline{x}).$$

Denote by $\varphi_v\colon \Sigma \cup \Sigma_*\to \R$ the function defined by 
$$\varphi_v({x}) = \prod_{j = 0}^{|v| - 1}\beta(\sigma_*^j vx).$$
Then
$$\beta_v(\underline{x}) = \frac{\varphi_v(\underline{x})}{\varphi_v(\sigma_* \underline{x})}\beta(\underline{x}).$$
It follows from the chain rule, for every $\underline{x}\in \Sigma$ 
\begin{equation}\label{3uyDerivadadeTv}
    \varphi_v(\underline{x}) = T_v'(\pi(\underline{x})),
\end{equation}
By \eqref{3uyDerivadadeTv} and \eqref{2DefBeta}, for all $\underline{x}\in \Sigma$,
$$\beta_v(\underline{x}) = \frac{T_v'(\pi(\underline{x}))}{f'(\pi(\underline{x})) \cdot T_{v}'(\pi(\sigma \underline{x}))} = \frac{T_v'(\pi(\underline{x}))}{f'(T_{x_1}(\pi(\sigma \underline{x}))) T_{v}'(\pi(\sigma \underline{x}))} = \frac{T_v'(\pi(\underline{x})) \cdot T_{x_1}'(\pi(\sigma \underline{x}))}{T_v'(\pi(\sigma \underline{x}))}.$$

Let $\gamma_v\colon \Sigma\cup \Sigma_* \to \R$ defined by 
\begin{equation}\label{3.3.2Gammav}
\gamma_v(\underline{x}) = -\log \beta_v(\underline{x}).
\end{equation}

It is clear by definition that for all $x\in \Sigma \cup \Sigma_*$
$$\gamma_v({x}) = \gamma(x) -\log \varphi_v({x}) + \log \varphi_v(\sigma_* {x} ).$$
Hence, $\gamma_v$ is cohomologous to $\gamma$.

On the other hand, by 
\begin{equation}\label{3Cohomologus}
    \gamma_v(\underline{x}) = \log(T_v'(\pi(\sigma \underline{x}))) - \log(T_v'(\pi(\underline{x}))) + \log f'(\pi( \underline{x} )).
\end{equation}

Let ${N}_v\colon \R\times \Sigma\cup \Sigma_* \to \R$ given by 
\begin{equation}
    {N}_v(t, \underline{x}) = \sum_{n = 0}^{\infty} \sum_{\substack{\underline{y}:\sigma_*^n\underline{y} = \underline{x} \\
    \sigma_*^j \underline{y} \neq \emptyset \text{ for } j < n}}1\{S_n\gamma_v(y) \leq t\}.
\end{equation}
And, as before for $\lambda> 0$, let 
$$\widetilde{N}_v(\lambda, \underline{x}) = {N}_v(-\log \lambda, \underline{x}).$$

\begin{lem}\label{3.3.2Almabda_v}
For all $0 < \lambda \leq \alpha_v$, we have 
$$\# A_\lambda^v= \widetilde{N}_v(\lambda/\alpha_v, \emptyset).$$
\end{lem}
\begin{proof}
Since for every finite word $u\in \Sigma_*$, with $|u| = n$ , we have
$$\prod_{j = 0}^{n - 1}  \beta_v(\sigma_*^j u) = \frac{\alpha_{vu}}{\alpha_v}.$$
Then, by the same reasoning as Lemma \ref{2Almabda},
\begin{align*}
    A_\lambda^v &= \bigcup_{n = 0}^\infty\{x_1\cdots x_n \in \Sigma_*: \alpha_{vx_1\cdots x_n} \geq \lambda\}\\
    &= \bigcup_{n = 0}^{\infty}\bigcup_{\substack{\substack{y: \sigma_*^n y = \emptyset \\ \sigma_* ^jy\neq \emptyset\\ j< n}}}\left\{y: \beta_v(y) \cdot \beta_v(\sigma_* y) \cdots  \beta_v(\sigma_*^{n - 1} y) \geq \lambda /\alpha_v\right\}.
\end{align*}
Since these unions are disjoint, this implies that 
$$\#A_\lambda^v = \sum_{n = 0}^{\infty} \sum_{\substack{\substack{y: \sigma_*^n y = \emptyset \\ \sigma_* ^jy\neq \emptyset\\ j< n}}}1\left\{ \beta_v(y) \cdot \beta_v(\sigma_* y) \cdots  \beta_v(\sigma_*^{n - 1} y) \geq \lambda/\alpha_v \right\}.$$

Finally, by the definition of $\gamma$, we have 
\begin{equation}\label{2EquivLogConBetav}
    \beta_v(y) \cdot\beta_v(\sigma_*y)\cdots \beta_v(\sigma^{n - 1}_*y) \geq \lambda/\alpha_v\iff S_n\gamma_v(y)\leq - \log(\lambda/\alpha_v).
\end{equation}
Therefore,
$$\# A_\lambda^v = \sum_{n = 0}^\infty \sum_{\substack{y: \sigma_*^n y = \emptyset\\ \sigma_* ^jy\neq \emptyset\\ j< n}} 1\{S_n \gamma_v(y) \leq -\log (\lambda/ \alpha_v)\} = \widetilde{N}_v(\lambda/\alpha_v, \emptyset).$$
\end{proof}

\begin{cor}
Assume that there is a symbol $1\in \mathcal{I}$ such that $T_1(0) = 0$. Let $v\in \Sigma_*$ and $0 < \lambda \leq \alpha_v$. Then
\begin{equation}
    \#(X_{\lambda} \cap T_v[0,1)) = \widetilde{N_v}(\lambda/\alpha_v, \emptyset) - \widetilde{N_v}\left(\lambda/\alpha_{v1}, 1\right).
\end{equation}    
\end{cor}
\begin{proof}
By Lemma \ref{3.2MedidaCilindro} and Lemma \ref{3.3.2Almabda_v} it will be enough to prove that 
$$\# \{w:vw1\in A_\lambda\} = \widetilde{N}_v\left(\lambda/\alpha_{v1}, 1\right).$$

Let $w\in \Sigma_n$ and set $y = w1$. Then $\sigma_*^n y = 1$. Moreover,
$$\prod_{j = 0}^{n - 1}\beta_v(\sigma_*^j y) = \frac{\alpha_{vw1}}{\alpha_{v1}}$$
Therefore 
$$vw1 \in A_\lambda \iff S_n\gamma_v(y) \leq -\log (\lambda / \alpha_{v1}).$$
Hence 
$$\#\{w: vw1 \in A_\lambda\} = \widetilde{N}_v(\lambda / \alpha_{v1}, 1).$$
\end{proof}

\begin{cor}\label{3CorollaryMeasureCylinder}
Assume that there is a symbol $1\in \mathcal{I}$ such that $T_1(0) = 0$. For all $\lambda \in (0, \alpha_v]$,
$$\mu_\lambda(T_v[0,1)) = \frac{\widetilde{N}_v(\lambda/ \alpha_v, \emptyset) - \widetilde{N}_v(\lambda/ \alpha_{v1}, 1)}{\# X_\lambda}.$$
\end{cor}

\subsection{Proof of Lemma \ref{2MainLem}}
The following lemma is a consequence of Bowen's formula for countable many branches (see \cite{MauldinUrbanski} and \cite{Iommi}).

\begin{lem}
For $\gamma\colon \Sigma\to \R$ defined in equation \eqref{2DefGamma} we have that 
$$P(-\gamma, \sigma) = 0.$$
\end{lem}

%From the renewal equation \eqref{3.3RenewalEquation} one sees that \ref{4RenewalGamma} implies 
%$$C(x) = \sum_{y: \sigma y = x}e^{-\gamma(y)} C(y).$$

Fix $v\in \Sigma_*$. If $v= \emptyset$, the conclusion is direct, so we may assume that $v\neq \emptyset$. Since $\gamma$ and $\gamma_v$ are cohomologous, and we are assuming that $\gamma$ is nonlattice, then $\gamma_v$ is also nonlattice. On the other hand, $-\log \varphi_v\in \mathcal{F}_\rho^b(\Sigma)$. Thus $\gamma_v$ satisfies the conditions of Corollary \ref{2Corollary: Cohomologus}. Hence, $\mu_{-\gamma} = \mu_{-\gamma_v}$ and
$$P(-\gamma, \sigma) = P(-\gamma_v, \sigma) = 0.$$
In what follows, we write 
$$\mu := \mu_{-\gamma} = \mu_{-\gamma_v} \quad\text{ and } \quad  \nu:= \nu_{-\gamma}.$$
Then $\gamma$
Moreover, we use the same notation $h_{-\gamma}$ and $h_{-\gamma_v}$ for the positive continuous extensions to $\Sigma\cup \Sigma_*$ given by Theorem \ref{2RenewalTheoremNLFiniteWords}.

Since $h_{-\gamma}$ and $h_{-\gamma_v}$ are the eigenfunctions of the Ruelle Perron Frobenius operator of $-\gamma$ and $-\gamma_v$ respectively, and
$$-\gamma_v = -\gamma + \log \varphi_v - \log \varphi_v \circ \sigma,$$
then the cohomology identities in \eqref{1.2EqCohomologus} imply
$$h_{-\gamma_v} = \frac{h_{-\gamma}}{\varphi_v} \cdot \left(\int \varphi_v \,d\nu_{-\gamma}\right).$$
For the dual eigenmeasures, we obtain
\begin{equation}\label{4EqNuGammaVConNuGamma}
    \nu_{-\gamma_v} = \frac{\varphi_v\cdot \nu_{-\gamma}}{\int \varphi_v\,d\nu_{-\gamma}}.
\end{equation}

Let $C(x) = C_{\gamma}(x)$ and $C_v(x) = C_{\gamma_v}(x)$ in \ref{2RenewalTheoremNLFiniteWords}.

\begin{proof}[Proof Lemma \ref{2MainLem}]
First assume that there is a unique symbol $1\in \mathcal{I}$ such that $T_1(0) = 0$.

Using Corollary \ref{3CorollaryMeasureCylinder}, we have 
%\ref{3CorollaryMeasureCylinder}, note that
\begin{align*}
    \lim_{\lambda \to 0^+}\mu_{\lambda}(T_v[0, 1)) &= \lim_{\lambda \to 0^+}\frac{\widetilde{N_v}(\lambda/ \alpha_v, \emptyset) - \widetilde{N_v}(\lambda/\alpha_{v1}, 1)}{\widetilde{N}( \lambda,  \emptyset) - \widetilde{N}(\lambda/\alpha_1, 1)}\\
    &= \lim_{\lambda \to 0^+}\frac{\alpha_v\cdot\frac{\lambda}{\alpha_v} \widetilde{N_v}(\lambda/ \alpha_v, \emptyset) - \alpha_{v1}\cdot\frac{\lambda}{\alpha_{v1}} \widetilde{N_v}(\lambda/\alpha_{v1}, 1)}{\lambda \widetilde{N}( \lambda,  \emptyset) -\alpha_1\frac{\lambda}{\alpha_1} \widetilde{N}(\lambda/\alpha_1, 1)}.
\end{align*}

Thus, by the Renewal Theorem \ref{2RenewalTheoremNLFiniteWords}
\begin{equation}
    \lim_{\lambda \to 0^+}\mu_{\lambda}(T_v[0, 1)) = \frac{\alpha_vC_v(\emptyset) - \alpha_{v1}C_{v}(1)}{C(\emptyset) - \alpha_1C(1)}.
\end{equation}
Since $\mu$ is $\sigma$-invariant we have 
$$\int \gamma_v\,d\mu = \int {\gamma}\,d\mu,$$
and by \eqref{3uyDerivadadeTv} 
$$\int \varphi_v\,d\nu = \int T_v'\circ \pi \,d\nu.$$ 
On the other hand, by the definition of $\varphi_v$, note that 
$$\varphi_v(\emptyset) = \alpha_v \quad \text{ and } \quad \varphi_v(1) = \frac{\alpha_{v1}}{\alpha_1}$$

Thus
$$\alpha_v C_v(\emptyset) = \frac{1}{\int \gamma\,d\mu} \cdot h_{-\gamma}(\emptyset) \cdot \int \varphi_v\,d\nu,$$
and 
$$\alpha_{v1} C_v(1) = \frac{\alpha_1}{\int \gamma\,d\mu} \cdot h_{-\gamma}(1) \cdot \int \varphi_v\,d\nu.$$
Consequently,
$$\frac{\alpha_vC_v(\emptyset) - \alpha_{v1}C_{v}(1)}{C(\emptyset) - \alpha_1C(1)} = \int \varphi_v\,d\nu = \int T_v' \circ \pi \,d\nu.$$

Let $\gamma^*\colon [0,1] \to \R$ be the map defined by
$$\gamma^*(x) = \log f'(x),$$
on the interiors of the partition elements, and define it to be zero at the partition endpoints. 

Note that $\gamma^*\circ \pi = \gamma$ away from the preimages of partition endpoints. This set of exceptional points is countable. Moreover, it has zero $\nu$-measure, because $\nu$ is not atomic. Indeed, for every $\underline{x}\in \Sigma$ and every $n\geq 1$, we have 
$$\nu(\{\underline{x}\}) = \int \mathcal{L}^n_{-\gamma}1_{\{ \underline{x} \}}\,d\nu = e^{-S_n \gamma(\underline{x})}\nu(\{\sigma^n x\}).$$
By the uniform contraction assumption and the fact that $\nu$ is a probability measure, we get 
$$\nu(\{\underline{x}\}) \leq c_M^n,$$
for every $n\in \N$. Letting $n \to\infty$, we get that $\nu(\{\underline{x}\}) = 0$.

Hence, the set of preimages of partition endpoints have zero mass for the conformal measure, so they do not affect the computation. Let $\nu_* = \nu \circ \pi^{-1}$, that is, 
$$\pi_*(\nu)(A) = \nu(\pi^{-1}(A)).$$
Then, for all functions $\varphi\colon [0,1] \to \R$,
$$\int \varphi\, d\pi_*(\nu) = \int \varphi\circ \pi \,d\nu.$$
Hence, if $\nu$ is a conformal measure for $-\gamma$, then $\pi_*(\nu)$ is a conformal measure for $-\gamma^*$ in the dynamical system $([0,1], f)$. In this system, we know by the change-of-variables formula that the Lebesgue measure is a conformal measure. This is because for all continuous functions $\varphi\in C([0,1])$,
\begin{align*}
  \int \varphi\,d\mathcal{L}_{-\gamma^*}^* \leb &= \int \mathcal{L}_{-\gamma^*} \varphi\,d\leb = \sum_{i\in \mathcal{I}}\int T_i'\cdot (\varphi\circ T_i)\,d\leb \\
  &= \sum_{i\in \mathcal{I}} \int_{T_i[0,1]} \varphi \,d\leb = \int\varphi\,d\leb.  
\end{align*}
Thus 
$$\mathcal{L}_{-\gamma^*}^* \leb = \leb.$$

By the uniqueness of the conformal probability measure for the potential $-\log |f'|$ for full-branch expanding Markov maps (\cite{MauldinUrbanzkiConformal}), we obtain
$$ \pi_*(\nu)=  \leb.$$

Thus 
$$\lim_{\lambda \to 0^+} \mu_{\lambda}( T_v([0,1)) )= \int T_v'\circ \pi\,d\nu = \int T_v' \,d\leb =T_v(1) - T_v(0) = \alpha_v.$$

Finally, if there is no symbol which fixes zero, we arrive at the same conclusion since 
$$\lim_{\lambda \to 0^+}\mu_{\lambda}(T_v[0, 1)) = \frac{\alpha_vC_v(\emptyset)}{C(\emptyset)} = \alpha_v.$$
\end{proof}

\subsubsection{Finite alphabet case}

To recover the original Kakutani, we now relate $\#\widehat{X_\lambda}$ with $\# A_\lambda$. Assume in this subsection that $\mathcal{I}$ is finite and $1\in \mathcal{I}$ is such that $T_1(0) = 0$.

\begin{cor}
For all $\lambda\in (0,1]$
\begin{equation}\label{3.2.1.RecoverCasoFinito}
    \#\widehat{X_\lambda}=  \#A_\lambda \cdot (\# \mathcal{I} - 1) + 1.
\end{equation}
\end{cor}
\begin{proof}
Every point of $\widehat{X}_\lambda$ is of the form $T_{vj}(0)$ with $v\in A_\lambda$ and $j\in \mathcal{I}$. Since the words ending with $1$ do not give new left endpoints, each point of $\widehat{X}_\lambda \setminus \{0\}$ has a unique representation of the form $T_{vj}(0)$, with $j\in \mathcal{I}\setminus \{1\}$. Therefore we conclude \eqref{3.2.1.RecoverCasoFinito}.
\end{proof}

Fix $v\in \Sigma_*\setminus \emptyset$.

\begin{cor}
Let $0 < \lambda \leq \alpha_v$. Then
\begin{equation}
    \#(\widehat{X}_{\lambda} \cap T_v[0,1)) = \widetilde{N}_v(\lambda/\alpha_v, \emptyset) \cdot (\#\mathcal{I} - 1)+ 1.
\end{equation}
\end{cor}
\begin{proof}
The points of $\widehat{X}_\lambda \cap T_v[0,1)$ except for $T_v(0)$, are in bijection with 
$$\{vwj: vw\in A_\lambda, j\in \mathcal{I}\setminus \{1\}\},$$
plus the point $T_v(0)$.

\end{proof}

\begin{cor}
For all $0 < \lambda \leq \alpha_v$,
$$\widehat{\mu_\lambda}(T_v[0,1)) = \frac{\widetilde{N}_v(\lambda/ \alpha_v, \emptyset)\cdot(\#\mathcal{I} - 1) + 1}{\# \widehat{X_\lambda}}.$$
\end{cor}

The same renewal argument now gives 
$$\lim_{\lambda \to 0^+} \widehat{\mu}_\lambda(T_v[0,1)) = \alpha_v.$$

\section{Comments on the lattice case}\label{SectionLattice}
In this case, we will restrict our attention to the case where $\mathcal{I}$ is finite. Let $\gamma$ be lattice regular potential with $P(-\gamma) = 0$.

\begin{teo}[{\cite[Corollary 3.6]{KombrinkRenewalGeneral}}]\label{3.4.2RenewalLatticeNormal}
Let $\gamma \in \mathcal{F}_{\rho}(\Sigma)$ be lattice function and let $\zeta, \psi\in \mathcal{F}_\rho(\Sigma)$ for some $\rho\in (0, 1)$ such that 
\begin{equation}\label{2Cohomologes}
    \gamma - \zeta = \psi - \psi \circ \sigma,
\end{equation}
and $\zeta(\Sigma) \subset a\Z$ for some $a> 0$ and it is not contained in any proper subgroup of $a\Z$. Furthermore, assume that $\gamma$ is not cohomologous to a function whose range lies in a proper subgroup of $a\Z$. Then, for $\theta\in \R$ we have
\begin{align*}
    \lim_{n \to\infty} e^{-an} N(an + \theta - \psi(\underline{x}), \underline{x}) &= \frac{a \cdot h_{-\zeta}(\underline{x})}{\int \zeta \,d\mu_{-\zeta}} \cdot \int_{\Sigma} \sum_{\ell = -\infty}^\infty e^{-a\ell} 1\{a\ell + \theta -\psi(\underline{y}) \geq 0\}\,d\nu_{-\zeta}(\underline{y})\\
    &= \frac{a \cdot h_{-\zeta}(\underline{x})}{\int \zeta\,d\mu_{-\zeta}\cdot (1 - e^{-a})} \cdot \int_{\Sigma} e^{-a\lceil (\psi(\underline{y}) - \theta)/ a \rceil}\,d\nu_{-\zeta}(\underline{y})\\
    &=\frac{a \cdot h_{-\zeta}(\underline{x})}{\int \zeta\,d\mu_{-\zeta}\cdot (1 - e^{-a})} \cdot \int_{\Sigma} e^{a\lfloor (\theta - \psi(\underline{y}))/ a \rfloor}\,d\nu_{-\zeta}(\underline{y}). 
\end{align*}
\end{teo}

Using the same strategy as in the proof of Theorem \ref{2RenewalTheoremNLFiniteWords}, we obtain the following lattice analogue.
\begin{teo}\label{3.4.2LatticeForFiniteWords}
Let $\gamma, \psi, \zeta \in \mathcal{F}_\rho(\Sigma \cup \Sigma_*)$ for some $\rho \in (0,1)$ satisfying \eqref{2Cohomologes} in $\Sigma \cup \Sigma_*$, $\zeta(\Sigma\cup \Sigma_*) \subseteq a\Z$ and $\gamma$ is not cohomologous to a function whose range lies in a proper subgroup of $a\Z$. Let $\theta \in \R$ be such that 
$$\nu_{-\zeta}(\{\underline{y} \in \Sigma:  \psi(\underline{y})\in \theta -a\Z\}) = 0.$$
Then for $x\in \Sigma \cup \Sigma_*$
\begin{align*}
   \lim_{n \to\infty} e^{-an} N_*(an + \theta - \psi({x}), {x}) &= \frac{a \cdot h_{*}(x)}{\int \zeta\,d\mu_{-\zeta}\cdot (1 - e^{-a})} \cdot \int_{\Sigma} e^{-a\lceil (\psi(y) - \theta)/ a \rceil}\,d\nu_{-\zeta}(y)\\
    &=\frac{a \cdot h_{*}(x)}{\int \zeta\,d\mu_{-\zeta}\cdot (1 - e^{-a})} \cdot \int_{\Sigma} e^{a\lfloor (\theta - \psi(y))/ a \rfloor}\,d\nu_{-\zeta}(y), 
\end{align*}
where $h_*$ is the unique positive function in $\mathcal{F}_\rho^b(\Sigma \cup\Sigma_*)$ such that $h_*|_\Sigma = h_{-\zeta}$ and, for $x\in \Sigma_*$
$$h_*(x) = \sum_{\substack{y : \sigma_* y= x\\  y \neq \emptyset}} e^{-\zeta(y)} h_*(y).$$
\end{teo}

To prove this we need to assume the following propositions proved in \cite{Lalley}.

\begin{lem}[{\cite[Lemma 6.1]{Lalley}}]
There is at most one nonnegative, continuous $h_*(x)$ on $\Sigma \cup \Sigma_*$ such that $h_*|_\Sigma = h_{-\zeta}$ and for $x\in \Sigma_*$
$$h_*(x) = \sum_{\substack{y: \sigma_*y = x\\ y\neq \emptyset}} e^{-\zeta(y)} h_*(y).$$
\end{lem}

\begin{lem}[{\cite[Lemma 6.2]{Lalley}}]\label{3.4.2LemmaLalleyCercano}
For each $\varepsilon > 0$ there exists $n_\varepsilon$ sufficiently large that if $x,x' \in \Sigma \cup \Sigma_*$ satisfy $x_i = x_i'$ for $i = 0, 1,\ldots, n_\varepsilon$ then
$$N_*(t,x) \leq N_*(t + \varepsilon, x').$$
\end{lem}

\begin{lem}\label{3.4.2ContinuityLatticeProfile}
Let $\theta \in \R$ be such that
$$\nu_{-\zeta}(\{\underline{y} \in \Sigma:  \psi(\underline{y}) \in \theta - a\Z\}) = 0.$$
Define the function $C_\psi \colon \R \to \R$ by
$$C_\psi(x) = \int_\Sigma e^{a\lfloor(x - \psi(\underline{y}))/a\rfloor}\, d\nu_{-\zeta}(\underline{y}).$$
Then $C_\psi$ is continuous at $\theta$.
\end{lem}

\begin{proof}
Consider a sequence $\{\theta_m\}_m$ such that $\theta_m \to \theta$. For every $y$ such that $\psi(y) \notin \theta - a\Z$, we have
$$\lfloor(\theta_m-\psi(y))/a\rfloor
\to
\lfloor(\theta-\psi(y))/a\rfloor.$$
By the no-atom assumption, this convergence holds for
$\nu_{-\zeta}$-almost every $y$. Since we are working with finite alphabet and $\psi$ is bounded, the functions
$$e^{a\lfloor(\theta_m-\psi(y))/a\rfloor}$$
are bounded by an integrable constant for all sufficiently large $m$. The result follows from the dominated convergence theorem.
\end{proof}

\begin{proof}[Proof of Theorem \ref{3.4.2LatticeForFiniteWords}]
Consider the renewal equation iteration:
\begin{equation}
N_*(t,x) = \sum_{\substack{y : \sigma_*^{k} y = x\\ \sigma_*^{k - 1} y \neq \emptyset}} N_*(t - S_k\gamma(y), y) + \sum_{m = 1}^{k - 1} \sum_{\substack{y : \sigma_*^m y = x\\ \sigma_*^{m - 1}y \neq \emptyset}}1\{t - S_m \gamma(y) \geq 0\} + 1\{t \geq 0\}.
\end{equation}
Fix $k$ large. For $t > k\|\gamma\|_\infty$ the last two sets of terms are bounded independently of $t$. Hence, after multiplying by $e^{-an}$ it does not contribute to the limit.
Thus, the asymptotic behaviour of $N_*(t, x)$ is completely determined by 
$$\sum_{\substack{y: \sigma_*^k y = x\\ \sigma^{k - 1}_*y \neq \emptyset}} N_*(t - S_k\gamma(y), y).$$
Observe that each $y$ in the preceding expression is a word of length of at least $k$, because $\sigma_*^{k - 1} y \neq \emptyset$. For each such $y$, there exists $y' \in \Sigma$ such that $y_i'  = y_i$ for $i = 0,1, \ldots, k -1$. By Lemma \ref{3.4.2LemmaLalleyCercano}, if $k \geq n_\varepsilon$ then
$$N_*(t - \varepsilon - S_k \gamma(y), y') \leq N_*(t - S_k\gamma(y), y) \leq N_*(t + \varepsilon - S_k\gamma(y), y').$$
Let 
$$t_n = an + \theta - \psi(x).$$
By \eqref{2Cohomologes}, 
$$\gamma - \zeta = \psi - \psi\circ \sigma_*.$$
Therefore, for $y$ such that $\sigma_*^ky = x$ 
$$t_n - S_k \gamma(y) = an + \theta - \psi(y) - S_k\zeta(y).$$
Then, assuming $k$ is large enough such that $|\psi(y) - \psi(y')| < \varepsilon$, we get
$$N_*(t_n - \varepsilon - S_k \gamma(y), y') \leq N_*(t_n - S_k\gamma(y), y) \leq N_*(t_n + \varepsilon - S_k\gamma(y), y').$$
Applying Theorem \ref{3.4.2RenewalLatticeNormal} to $y'$ gives the desired lattice factor, with $\theta $ replaced by $\theta \pm \varepsilon$. By Lemma \ref{3.4.2ContinuityLatticeProfile}, the map
$$\theta\mapsto
\int_\Sigma e^{a\lfloor(\theta-\psi(y))/a\rfloor}\,
d\nu_{-\zeta}(y)$$
is continuous at $\theta$. 
Therefore, letting $\varepsilon\to0$ in the upper and lower bounds gives the stated expression.

We now study the remaining coefficient
$$\sum_{\substack{y: \sigma_*^k y = x\\ \sigma_*^{k-1}y\neq  \emptyset}} e^{-S_k\zeta(y)} h_{*}(y'),$$
where $h_*$ was defined as in the nonlattice case.

By continuity of $h_*$, $h_*(y')$ converges uniformly to $h_*(y)$. So when $k\to\infty$ we get that this remaining coefficient has the same limit as 
$$\sum_{\substack{y: \sigma_*^k y = x\\ \sigma_*^{k-1}y\neq  \emptyset}} e^{-S_k\zeta(y)} h_{*}(y) = h_*(x),$$
where this equality follows by iterating the defining relation for $h_*$. 
\end{proof}

In order to study the lattice case, we first characterize the lattice geometric potentials.

\begin{teo}\label{5LatticeCharact}
Let $\mathcal{P}$ be the partition generated by $\{T_i\}_{i \in \mathcal{I}}$ in the sense of Definition \ref{1.1DefPartitionGeneratedbyTi}. Suppose that the geometric potential $\gamma\colon \Sigma \to \R$ is lattice and $\mathcal{I}$ is finite. Then there exists a partition $\mathcal{P}_\ell$ generated by $\{T_i^\ell\}_{i\in \mathcal{I}}$, where each $T_i^\ell$ is an affine linear function and a diffeomorphism $g\in \mathcal{C}^{1 + \alpha}([0,1])$ such that for each $i\in \mathcal{I}$ we have 
$$T_i = g \circ T_i^\ell \circ g^{-1}.$$
\end{teo}
\begin{proof}
Assume that $\mathcal{I} = \{1,\ldots, m\}$. The proof is divided into three main steps:

\textbf{Step 1}: $\gamma$ is cohomologous to a function $\gamma_\ell$ which is constant on cylinders of length 1.

Indeed, since $\gamma$ is lattice, it is cohomologous to a function $\zeta\colon \Sigma \to a\Z$ and the image is not contained in a proper subgroup of $a\Z$. Take $\varphi\colon \Sigma \to \R$ a Hölder continuous function such that 
\begin{equation}\label{5CohomologousFunction}
  \gamma = \zeta + \varphi - \varphi \circ \sigma.  
\end{equation}
Since $\zeta$ is continuous, then $\zeta$ is locally constant. Since $\mathcal{I}$ is finite, $\Sigma$ is compact. Thus, $\zeta$ is constant in cylinders of length $M$.

Let $b_j$ be the common endpoint between $I_j$ and $I_{j +1}$ for $j = 1,\ldots, m -1$. Then, this common endpoint has two symbolic codings 
\begin{equation}\label{5EndpointsCoding}
    x_j = jmmm\ldots, \quad y_j = (j + 1) 111\cdots.
\end{equation}
We have 
$$\pi(x_j) = \pi (y_j) = b_j.$$
Because the branches are $\mathcal{C}^1$, for every word $w$ and every symbol $i$,
$$\gamma(iwx_j) = \gamma(iwy_j).$$
Take a word $u$ of length $M - 1$, we have 
$$\zeta(iux_j) = \zeta(iuy_j).$$
Then, using \eqref{5CohomologousFunction}, we have that for all $i\in \mathcal{I}$
$$\varphi(iux_j) - \varphi(iuy_j) = \varphi(u x_j) - \varphi(uy_j).$$
Inductively, for all words $w\in \Sigma_*$
$$\varphi(wux_j ) - \varphi(wuy_j) = \varphi(u x_j) - \varphi(uy_j).$$
Since this is true for every word, making the length arbitrarily long, we get by continuity of $\varphi$ that this implies that for all words $u$ of length $M- 1$:
\begin{equation}\label{5VarphiM-1}
    \varphi(ux_j) = \varphi(uy_j).
\end{equation}

Take a word $u'$ of length $M -2$. The equality above gives, for every symbol $i$,
$$\gamma(iu'x_j) = \gamma(iu'y_j).$$
On the other hand, by \eqref{5CohomologousFunction},
$$\zeta([iu'j]) - \zeta([iu'(j + 1)]) = \varphi(u'x_j) - \varphi(u' y_j),$$
because the term $\varphi(iu' x_j) - \varphi(iu'y_j)$ vanishes by \eqref{5VarphiM-1}.

We just proved that $\zeta([iu'j]) - \zeta([iu'(j + 1)])$ is independent of the initial symbol $i\in \mathcal{I}$.

We now prove that this implies that $\zeta$ is cohomologous to a function $\zeta_{M - 1}$ which is constant in cylinders of length $M - 1$. Define $\Phi\colon \Sigma \to \R$ locally constant on cylinders of length $M- 1$, defined by
$$\Phi([u'j]) = \zeta([1u'j]) - \zeta([1u'1]),$$
for any $u'\in \Sigma_{M - 2}$ and $j\in \mathcal{I}$.
We claim that for every $i \in \mathcal{I}$ the function $\zeta ([iu'j]) - \Phi([u'j])$ is independent of $j$. Indeed, consider $j\neq k\in \mathcal{I}$, then 
$$\zeta([iu'j]) - \Phi([u'j]) - \zeta([iu'k]) + \Phi([u'k])  = \underbrace{\zeta([iu'j]) - \zeta([iu'k])}_{\text{independent of }i} - \underbrace{(\zeta([1u'j]) - \zeta([1u'k]))}_{\text{independent of }1},$$
by the independence of the first term, we conclude that this difference is zero. Thus,
$$\zeta_{M - 1} = \zeta + \Phi - \Phi\circ \sigma,$$
is a function depending only on the first $M - 1$ symbols cohomologous to $\zeta$. Repeating this argument, we get that $\zeta$ is cohomologous to a function $\gamma_{\ell}$ constant in cylinders of length 1.

\textbf{Step 2}: Construct $T_i^\ell$. Since $\gamma_\ell$ is cohomologous to $\gamma$, we get that $P(-\gamma_\ell) = P(-\gamma) = 0$. Thus,
$$\sum_{i \in \mathcal{I}} \exp(- \gamma_\ell([i])) = 1.$$
Let 
$$\alpha_i^\ell = \exp(- \gamma_\ell([i])).$$
Let $\mathcal{P}_\ell = \{I_1^\ell,\ldots, I_m^\ell\}$ be the partition of $[0,1]$ ordered as $\mathcal{P}$ such that $I_i^\ell = \alpha_i^\ell$ for all $i\in \mathcal{I}$. Let $T_i^\ell$ the affine contractions sending $[0,1]$ to $I_i^\ell$. From the construction, we have that 
$$\gamma_{\mathcal{P}_\ell} = \gamma_\ell.$$
Let $f_\ell$ be the map defined by $f_{\mathcal{P}_\ell}$ in  \eqref{1.1DefiSistemaDinamico0,1}.

\textbf{Step 3}: Construct the conjugating map $g\colon [0,1] \to [0,1]$.
So far, we know that for each $\underline{x}\in \Sigma$, we have that there exists $\Psi\colon \Sigma \to \R$ such that 
\begin{equation}\label{5EqCohomologuesLevantada}
\log (f'(\pi(\underline{x}))) - \log (f'_\ell(\pi_\ell(\underline{x}))) = \Psi (\underline{x}) - \Psi(\sigma \underline{x}).    
\end{equation}
We will construct a map $\psi\colon [0,1] \to \R$, such that $\Psi = \psi \circ \pi$. It will be enough to prove that 
$$\pi(\underline{x}) = \pi (\underline{y}) \implies \Psi(\underline{x}) = \Psi (\underline{y}).$$
Recall the definition of the endpoints, $x_j$ and $y_j$ from \eqref{5EndpointsCoding} whose preimages will be the ones that give the same coding. We have for any word $w\in \Sigma_*$,
$$\pi_\ell (wx_j) =\pi_\ell (wy_j) \quad \text{ and } \quad \pi(wx_j) = \pi(wy_j).$$
This implies that
$$(\Psi - \Psi \circ \sigma)(wx_j) = (\Psi - \Psi \circ \sigma)(wy_j).$$
Iterating this, we get that for any word 
$$\Psi (x_j) - \Psi (y_j) = \Psi(wx_j) - \Psi(wy_j).$$
By continuity of $\Psi$, this implies that for all $j$, $\Psi(x_j) = \Psi(y_j)$, and thus for any word $w\in \Sigma_*$,
$$\Psi(wx_j) = \Psi(wy_j).$$
Since these are the only cases where $\pi(\underline{x}) = \pi(\underline{y})$, we get the equality. Let $\psi \colon[0,1] \to \R$ be the continuous map such that $\Psi = \psi \circ \pi$. We will prove now that $\psi\colon [0, 1]\to \R$ is Hölder.

Let $s,t\in [0,1]$, and consider $n$ the unique natural number such that 
$$c^{n + 1} < |s - t| \leq c^n,$$
where $c$ is the bound shown in Definition \ref{1.1DefPartition}. We will consider two cases.
\begin{itemize}
    \item \textbf{Case 1}: both $s$ and $t$ are contained in a cylinder set $T_v[0,1]$ for some $v\in \Sigma_n$. In this case, using the Hölder continuity of $\Psi$, we have 
    $$|\psi(s) - \psi(t)| \leq C\rho^n.$$
    \item \textbf{Case 2}:If $s$ and $t$ are contained in different $n$-level cylinder sets, we have using the triangular inequality and the choice of $n$ that
    $$|\psi(s) - \psi(t)| \leq 2C\rho^n.$$
\end{itemize}
Let $\beta$ such that 
$$c^\beta =\rho.$$
Since $c^{n + 1}<|s-t|$ we have 
$$|\psi(s)- \psi(t)| \leq 2C \rho^n \leq 2C c^{-\beta}|s-t|^\beta.$$
Hence, $\psi$ is Hölder.

Let's construct the conjugation function now. By \eqref{5EqCohomologuesLevantada}, we get that for each $s$
$$-\log T_i'(s) + \log \alpha_i^\ell =\psi(T_i(s)) - \psi(s)\implies T_i'(s) =\alpha_i^\ell \cdot \exp(-\psi(T_i(s)) + \psi(s)).$$
Define 
$$G (t) = \frac{\int_0^t \exp(\psi (s))\,ds}{\int_0^1\exp(\psi(s)) \,ds}.$$
We clearly have $G(0) = 0$ and $G(1) = 1$. Also we know that $G$ is bijective and we have
\begin{equation}\label{5EquationDerivatives}
    (G\circ T_i)'(t) = \alpha_i^\ell G'(t) = (T_i^\ell \circ G)'(t)
\end{equation}
for all $i$. 

Since $G(1) - G(0) = 1$, we get 
$$G\circ T_1 = T_1^\ell \circ G.$$
In particular,
$$G(T_1(1)) = T_1^\ell (G_1(1)).$$
Since the intervals are ordered,
$$T_1(1) = T_2(0) \quad \text{ and }\quad T_1^\ell (1) = T_2^\ell (0).$$
Repeating this argument inductively, we obtain $G\circ T_i = T_i^\ell \circ G$ for every $i \in \mathcal{I}$.
Thus, defining $g = G^{-1}$, we get By \eqref{5EquationDerivatives} holds for every $i\in \mathcal{I}$, it follows that 
$$G\circ T_i=  T_i^{\ell}\circ G,$$
for every $i\in \mathcal{I}$. Putting $g= G^{-1}$, we get 
$$T_i \circ g = g\circ T_i^\ell.$$
\end{proof}

The preceding theorem shows that, in the finite alphabet case, every lattice
geometric potential arises from a smooth conjugacy of an affine model. Therefore, to understand the obstruction in the lattice case, it is enough to
study systems of the following form.

Let $\mathcal{P}$ be the partition generated by $\{T_i\}_{i \in \mathcal{I}}$, where $\gamma := \gamma_{\mathcal{P}}$ is a lattice function. Let $\{T_i^\ell\}_{i \in \mathcal{I}}$ be the transformations constructed in Theorem \ref{5LatticeCharact} and $g\colon [0,1] \to [0,1]$ the $C^{1 + \alpha}$ diffeomorphism such that $T_i = g\circ T_i^\ell \circ g^{-1}$. 

As before in Section \ref{SectionPrelim}, construct $f_{\ell}, \gamma_{\ell}, \beta_{\ell}, \pi_{\ell}$ and $N^{\ell}$ associated to the partition generated by $\{T_i^\ell\}_{i \in \mathcal{I}}$.

We have 
$$\{\pi(\underline{x})\} = \bigcap_{n = 1}^\infty g\circ T_{x_1\ldots x_n}^\ell \circ g^{-1} [0,1] = g\left(\bigcap_{n = 1} T^\ell_{x_1\ldots x_n}[0,1]\right) = \{g(\pi_{\ell}(\underline{x}))\}.$$
With this construction, by the chain rule 
\begin{align*}
f' (\pi (\underline{x})) &= g'(f_{\ell}(\pi_{\ell}(\underline{x}))) \cdot f'_{\ell}(\pi_{\ell}(\underline{x})) \cdot \frac{1}{g'(\pi_{\ell}(\underline{x}))} \\&= g'(\pi_{\ell}(\sigma\underline{x})) \cdot f'_{\ell}(\pi_{\ell}(\underline{x})) \cdot \frac{1}{g'(\pi_{\ell}(\underline{x}))}.    
\end{align*}

Define $\psi \colon \Sigma \to \R$ by 
$$\psi(\underline{x}) = g'(\pi_\ell(\underline{x})).$$
Then
$$\gamma = \gamma_\ell + \log\psi \circ \sigma - \log\psi.$$

Note that in this case $\psi \colon \Sigma \cup \Sigma_* \to \R$ is defined by 
$$\psi(x_1\ldots x_n) = \frac{g(T_{x_1\ldots x_n}^\ell (1)) - g(T^\ell_{x_1\ldots x_n }(0))}{T^\ell_{x_1\ldots x_n} (1) - T^\ell_{x_1\ldots x_n}(0)} \quad \text{ and } \quad \psi(\emptyset) = 1.$$
This definition is chosen so that the finite-word multiplicative identities for $\beta$ remain valid.

Thus, for $\underline{x}\in \Sigma \cup \Sigma_*$
$$\beta (\underline{x}) = \beta_{\ell}(\underline{x}) \cdot \frac{\psi(\underline{x})}{\psi(\sigma \underline{x})} \implies \gamma(\underline{x}) = \gamma_{\ell}(\underline{x}) - \log \psi(\underline{x}) + \log \psi(\sigma\underline{x}).$$

Let $\Psi\colon \Sigma \cup \Sigma_* \to \R$, defined by
$$\Psi = -\log \psi.$$
Note that 
$$\Psi(\emptyset) = 0 \quad \text{ and }\quad \Psi(1) = -\log g(\alpha_1^\ell) + \log\alpha_1^\ell $$

Let $\gamma_v\colon \Sigma \cup \Sigma_* \to \R$ as \eqref{3.3.2Gammav}. Let 
$$\Psi_v:=\Psi-\log\varphi_v.$$
Then
$$\gamma_v - \gamma_{\ell} = \Psi - \log \varphi_v  - (\Psi - \log \varphi_v) \circ \sigma = \Psi_v - \Psi_v\circ \sigma.$$

We have 
$$\Psi_v(\emptyset) = -\log \alpha_v \quad \text{ and } \quad \Psi_{v}(1) = -\log g(\alpha_1^\ell) + \log \alpha_1^\ell - \log \alpha_{v1} + \log \alpha_{1}.$$

In order to use Theorem \ref{3.4.2LatticeForFiniteWords} as it was used in the proof of Lemma \ref{2MainLem}, we want to compute the limit as $n\to \infty$ of the expression$$\frac{N_v ( an + \log \alpha_v, \emptyset ) - N_v (an + \log \alpha_{v1}, 1)}{N(an, \emptyset) - N(an + \log \alpha_1, 1)}.$$
Let 
$$\theta_1 = -\log\left(\frac{g(\alpha_1^\ell)}{\alpha_1^\ell \cdot \alpha_1}\right) = \log \alpha_1^\ell.$$

Considering each term separately, we get 
\begin{align*}
N_v (an + \log \alpha_v, \emptyset) &= N_v(an - \Psi_v(\emptyset), \emptyset) \\
N_v(an + \log \alpha_{v1}, 1) &= N_v( an - \Psi_v(1) + \theta_1, 1) \\
N(an, \emptyset) &= N(an- \Psi(\emptyset), \emptyset)\\
N(an + \log \alpha_1, 1) &= N(an - \Psi(1) + \theta_1, 1).
\end{align*}

In order to apply Theorem \ref{3.4.2LatticeForFiniteWords}, we assume that
$$\nu_{-\gamma_\ell}
\bigl(\{\underline{y} \in\Sigma:\Psi(\underline{y})\in a\mathbb Z\}\bigr)=0 \quad\text{ and } \quad \nu_{-\gamma_\ell}
\bigl(\{\underline{y}\in\Sigma:\Psi_v(\underline{y})\in a\mathbb Z\}\bigr)=0.$$
Since $\theta_1\in a\mathbb Z$, these are the only residue classes that occur in the computation below.

Under this assumption, the same computation as in the nonlattice case and the fact that $h_{-\gamma_\ell}\equiv 1$ give
$$\lim_{n\to\infty}\frac{N_v ( an + \log \alpha_v, \emptyset ) - N_v (an + \log \alpha_{v1}, 1)}{N(an, \emptyset) - N(an + \log \alpha_1, 1)} = \frac{\int e^{a\lfloor -\Psi_v(y)/a \rfloor} - e^{a\lfloor (\theta_1 - \Psi_v(y)) / a\rfloor }\,d\nu_{-\gamma_\ell}(y)}{\int e^{a\lfloor -\Psi(y)/a \rfloor} - e^{a\lfloor (\theta_1 - \Psi(y)) / a\rfloor }\,d\nu_{-\gamma_\ell}(y)}.$$

Since for all $\theta$ 
$$\lfloor \theta \rfloor  = \theta - \{\theta\},$$
we can replace this and get that 
\[
\begin{aligned}
&e^{a\lfloor -\Psi_v(y)/a \rfloor}
-
e^{a\lfloor (\theta_1-\Psi_v(y))/a \rfloor} \\
&=
e^{-\Psi_v(y)-a\{-\Psi_v(y)/a\}}
\left(
1-
e^{\theta_1
-a\{(\theta_1-\Psi_v(y))/a\}
+a\{-\Psi_v(y)/a\}}
\right).
\end{aligned}
\]
Since $\theta_1\in a\Z$, we get
$$e^{a\lfloor -\Psi_v(y)/a \rfloor} - e^{a\lfloor (\theta_1 - \Psi_v(y)) / a\rfloor } = e^{-\Psi_v(y) - a \{-\Psi_v(y)/ a\}} \left(1 - e^{\theta_1  - a \{(-\Psi_v(y))/a\} + a\{-\Psi_v(y)/a\}}\right).$$

Note that 
\begin{align*}
\Psi_v(y) &= \Psi(y) + \sum_{i = 0}^{|v| - 1} \gamma(\sigma^iv y)\\
&= \Psi(y) + \sum_{i = 0}^{|v| - 1} \gamma_\ell(\sigma^i vy) + \Psi(vy) - \Psi(y)\\
&= \sum_{i = 0}^{|v| - 1} \gamma_\ell(\sigma^i vy) + \Psi(vy).
\end{align*}

Since $\gamma_\ell(y)\in a\Z$ for all $y\in \Sigma$, we have 
$$e^{a\lfloor -\Psi_v(y)/a \rfloor} - e^{a\lfloor (\theta_1 - \Psi_v(y)) / a\rfloor }  = \alpha_v^\ell \cdot e^{-\Psi(vy) - a \{-\Psi(vy)/ a\}} \left(1 - e^{\theta_1}\right).$$

Let 
$$I_v := \int e^{-\Psi(vy) - a \{-\Psi(vy)/ a\}} \left(1 - e^{\theta_1}\right)\,d\nu_{-\gamma_\ell}(y),$$
and 
$$I := \int e^{-\Psi(y) - a \{-\Psi(y)/ a\}} \left(1 - e^{\theta_1}\right)\,d\nu_{-\gamma_\ell} (y).$$

Recall that for $\underline{x}\in \Sigma$
$$\Psi(\underline{x}) = -\log \psi(\underline{x}) = -\log g' (\pi_\ell(\underline{x})).$$
Let 
$$F(t) = g'(t) e^{-a\{\log g'(t) /a\}}\left(1 - e^{\theta_1 }\right).$$

Then 
$$I = \int F\circ \pi_\ell\,d\nu_{-\gamma_\ell} = \int F\,d\leb,$$
and 
$$I_v = \int F\circ T_v^\ell\circ \pi_\ell\, d\nu_{-\gamma_\ell} = \int F\circ T_v^\ell \,d\leb.$$
We would like to have that 
$$\frac{I_v}{I} = \frac{\alpha_v}{\alpha_v^\ell}.$$

However, the fractional-part terms are not invariant under the smooth change of variables $g$. Consequently, the ratio $I_v/I$ need not coincide with the geometric scaling factor $\alpha_v/\alpha_{v}^\ell$.

The following example shows that we could get arbitrary $C^\infty$ small perturbations of a linear Kakutani system where we don't get equidistribution. This example was found wit the assistance of ChatGPT (GPT - 5.5 Thinking) and subsequentially verified by the author.

\begin{exam}\label{4Theorem: Example4.1}
 Consider the doubling map:
$$T_1^\ell (x) = \frac{x}{2}, \quad T_2^\ell(x) = \frac{x + 1}{2}.$$
Then $\gamma_\ell$ is the constant function $\log 2$, which is clearly lattice.

Let $\varepsilon \in (0,1/4)$ and set 
$$g_\varepsilon(x) = x +\frac{\varepsilon}{2\pi} (1 - \cos (2\pi x)).$$
Then 
$$g'_\varepsilon(x) = 1 + \varepsilon \,\sin (2\pi x) > 0.$$
Define for $i =1,2$
$$T_i = g_\varepsilon \circ T_i^\ell \circ g_\varepsilon^{-1}.$$

For every word $v\in \Sigma_k$, we get 
$$\leb T_v[0,1] \in \left[(1 - \varepsilon) 2^{-k}, (1 + \varepsilon) 2^{-k}\right].$$
Set $J = [0, g_\varepsilon(1/2))$. Then, for every word $w\in \bigcup_{j =0}^{n - 1} \Sigma_j$, we have 
$$\leb (T_w [0,1]) \geq (1 - \varepsilon) 2^{- (n - 1)} > 2^{- n},$$
for $|w| \geq n+ 1$
$$\leb (T_w[0, 1]) \leq (1 + \varepsilon)2^{- (n + 1)} < 2^{-n}.$$
For $|w| = n$, we get $\leb( T_w^\ell [0,1]) = 2^{-n}$, then its image under $g_\varepsilon$ has Lebesgue measure greater than $2^{-n}$ precisely when $w$ starts with $1$ since $g_\varepsilon' > 1$ on the first half. In other words,
$$X_{2^{-n}} = \left\{g_\varepsilon \left(\frac{k}{2^{n - 1}}: 0\leq k < 2^{n - 1} \right) \right\} \cup \left\{g_\varepsilon\left(\frac{2k + 1}{2^n} \right): 0 \leq k < 2^{n-2} \right\}.$$
Thus
$$\# X_{2^{-n}} = 2^{n- 1} + 2^{n- 2} \quad \# (X_{2^{- n}} \cap J) = 2^{n - 2} + 2^{n-2}.$$
Consequently,
$$\mu_{2^{-n}} (J) = \frac{2}{3}.$$
If we take now 
$$\lambda_n = \frac{3}{4} \cdot 2^{-n},$$
we get 
$$X_{\lambda_n} = \left\{g \left(\frac{k}{2^n}\right): 0 \leq k < 2^n  \right\} \quad \text{ and } \quad \mu_{\lambda_n} (J) = \frac{1}{2},$$
since 
$$\frac{1 + \varepsilon}{2} < \frac{3}{4} < 1 - \varepsilon.$$

\end{exam}

This proves failure of Lebesgue equidistribution in this particular case. Furthermore, since 
$$(g_\varepsilon - \text{Id})(x) = \frac{\varepsilon}{2\pi} \left( 1 - \cos(2\pi x) \right),$$
this example also proves that there is an arbitrarily $C^\infty$-small nonlinear perturbation of the $\alpha = 1/2$ Kakutani system for which equidistribution fails, even though the affine system is equidistributed.

\bibliographystyle{amsplain}
\bibliography{BibFile}

\end{document}